\documentclass[11pt]{amsart}
\usepackage{amsmath, amsthm, amssymb}
\usepackage{graphicx}
\usepackage{xcolor}
\usepackage{algorithm}
\usepackage{algpseudocode}
\usepackage[colorlinks,linkcolor=blue,citecolor=blue,urlcolor=magenta,linktocpage,plainpages=false]{hyperref}
\usepackage{caption}
\usepackage{subcaption}
\usepackage{setspace}
\usepackage[left=1in, right=1in, top=1in, bottom=1in]{geometry}
\usepackage{changepage}
\usepackage{enumitem}

\newtheorem{theorem}{Theorem}
\newtheorem{lemma}[theorem]{Lemma}
\newtheorem{corollary}[theorem]{Corollary}
\newtheorem{proposition}[theorem]{Proposition}

\newtheorem{conjecture}[theorem]{Conjecture}
\numberwithin{theorem}{section}

\theoremstyle{definition}
\newtheorem{example}[theorem]{Example}
\newtheorem{definition}[theorem]{Definition}

\theoremstyle{remark}
\newtheorem{remark}[theorem]{Remark}
\newtheorem{observation}[theorem]{Observation}

\newcommand{\I}{\mathcal{I}}
\newcommand{\R}{\mathbb{R}}
\newcommand{\conv}{\textup{conv}}
\renewcommand{\span}{\textup{span}}

\newcommand{\im}{\textup{image }}

\newcommand{\hf}{f^h}
\newcommand{\hH}{\hat{H}}

\newcommand{\hx}{\hat{x}}
\newcommand{\hy}{\hat{y}}
\renewcommand{\P}{\mathcal{P}}

\newcommand{\blue}[1]{{\color{black} #1}}

\graphicspath{{./}}

\title{Aggregations of quadratic inequalities and  hidden hyperplane convexity}

\author{Grigoriy Blekherman}\address[Grigoriy Blekherman]{School of Mathematics, Georgia Institute of Technology, Atlanta, GA, USA}\email{greg@math.gatech.edu}
\author{Santanu S. Dey}\address[Santanu S. Dey]{School of Industrial and Systems Engineering, Georgia Institute of Technology, Atlanta, GA, USA}\email{santanu.dey@isye.gatech.edu}
\author{Shengding Sun}\address[Shengding Sun]{School of Mathematics, Georgia Institute of Technology, Atlanta, GA, USA}\email{ssun313@gatech.edu}

\date{\today}
\thanks{Grigoriy Blekherman and Shengding Sun were partially supported by NSF grant DMS-1901950.}

\begin{document}
	\maketitle

	\begin{abstract}
		
		We study properties of the convex hull of a set $S$ described by quadratic inequalities. A simple way of generating inequalities valid on $S$ is to take nonnegative linear combinations of the defining inequalities of $S$. We call such inequalities \emph{aggregations}. Special aggregations naturally contain the convex hull of $S$, and we give sufficient conditions for intersection of such aggregations to define the convex hull. We introduce the notion of hidden hyperplane convexity (HHC), which is related to the classical notion of hidden convexity of quadratic maps. We show that if the quadratic map associated with $S$ satisfies HHC, then the convex hull of $S$ is defined by special aggregations. 
		To the best of our knowledge, this result generalizes all known results regarding aggregations defining convex hulls. Using this sufficient condition, we are able to recognize previously unknown classes of sets where aggregations lead to convex hull. We show that the condition known as positive definite linear combination for every triple of inequalities, together with hidden hyperplane convexity is sufficient for finitely many aggregations to define the convex hull, answering a question raised in~\cite{dey2021obtaining}. All the above results are for sets defined using open quadratic inequalities. For closed quadratic inequalities, we prove a new result regarding aggregations giving the convex hull, without topological assumptions on $S$, which were needed in~\cite{ModaresiV2015,dey2021obtaining}. 
		
		
	\end{abstract}

	\section{Introduction}
	The well-known Farkas lemma in linear programming states that any implied linear inequality for a non-empty set defined by finitely many linear inequalities, can be obtained by taking a nonnegative weighted combination of the original inequalities. We call the procedure of obtaining implied inequalities for a given set by rescaling the defining constraints by nonnegative weights and then adding the scaled constraints together as \emph{aggregation}. Aggregations have also been studied in the context of integer linear programming (for example,~\cite{bodur2018aggregation}) and mixed-integer nonlinear programming (for example,~\cite{gleixner2020generalized}) to obtain better cutting-planes or improved dual bounds. In this paper, we extend the study of aggregation~\cite{yildiran2009convex,Burer2016,ModaresiV2015,dey2021obtaining} in the context of quadratic constraints. While sets defined by linear inequalities are always convex, sets defined by quadratic inequalities are usually not, and we address the question of when the convex hull can be found via aggregation.
	
	Let $S\subseteq \mathbb{R}^n$  be a set defined by finitely many quadratic constraints. Since we are interested in finding the convex  hull of $S$, it makes sense to consider only ``good aggregations'', which have at most one negative eigenvalue, so that the set defined by the aggregated constraint has at most two connected components that are both convex, and furthermore contains the convex hull in one of its connected components. It is known that the convex hull of $S$ is always described by intersection of these ``good aggregations'' in the case where $S$ is defined by two quadratic constraints \cite{yildiran2009convex}. In fact, the convex hull of a set defined by two quadratic inequalities can be obtained as the intersection of two good aggregations. 
	Henceforth, for simplicity, if it is clear from context we will drop the term ``good" and refer to good aggregation as aggregation. 
	
	The paper~\cite{dey2021obtaining} extended the result of~\cite{yildiran2009convex} to the case of a set $S$ defined using three quadratic constraints, by showing that under an additional condition called \emph{positive definite linear combination} (PDLC), the convex hull of $S$ is obtained as the intersection of good aggregations. They also show via examples that if PDLC does not hold, then the convex hull of $S$ may not be given by good aggregations. 
	
	A key ingredient of the result in~\cite{yildiran2009convex} is the S-lemma~\cite{polik2007survey}. The main use of the PDLC condition in \cite{dey2021obtaining} is also to prove a version of the homogeneous S-lemma for three  quadratics under PDLC. The result is an application of Calabi's convexity theorem \cite{calabi1982linear,polyak1998} which states that under PDLC the image of $\mathbb{R}^n$ by three quadratic forms is a closed convex cone in $\R^3$. Convexity of the image of quadratic maps is a classical mathematical problem in convex geometry and real algebraic geometry, dating back to Dines' theorem \cite{dines} and Brickman's theorem \cite{brickman}. Yakubovich's S-lemma (also called S-procedure) \cite{yakubovich1977} connects this problem to the realm of polynomial optimization. Since then the notion of hidden convexity has become a powerful tool to study quadratic programming \cite{polyak1998,hiriart2002backnforth}. See \cite{sheriff2013convexity} for a mathematical treatment of convexity of quadratic maps and \cite{polik2007survey} for a survey of the S-lemma. 
	
	The main goals of this paper are three-fold:
	\begin{itemize}
		\vspace{-.2cm}
		\item General sufficient conditions for aggregations to yield convex hull:  We establish a new sufficient condition for aggregations to define the convex hull, which we call \emph{hidden hyperplane convexity (HHC).} To the best of our knowledge, hidden hyperplane convexity gives the most general result on convex hull of a region defined by quadratic inequalities being given by aggregations. In particular, we simultaneously generalize the results of \cite{yildiran2009convex} and \cite{dey2021obtaining}, which deal with two and three quadratic inequalities respectively. Furthermore, we give new examples of sets described by more than $3$ quadratic inequalities where convex hull is given by aggregations. We show that hidden hyperplane convexity is a stronger requirement than hidden convexity, and in order for the convex hull of a set defined by quadratic inequalities to be given by aggregations, hidden convexity is not sufficient while HHC is not necessary.\\

		
		\item Finiteness of aggregations: While \cite{yildiran2009convex} shows that only two good aggregations suffice to define the convex hull in the case of sets described by two quadratic inequalities, the paper~\cite{dey2021obtaining} only shows that the intersection of good aggregations yields the convex hull --leaving the question of whether only a finite number of good aggregations are sufficient to obtain the convex hull for three quadratic constraints satisfying PDLC, as an open problem. We answer this question in the affirmative in this paper, and we show that six aggregations suffice to describe the convex hull for three quadratic constraints satisfying PDLC. Furthermore, we establish a more general sufficient condition for finiteness of aggregations in Theorem~\ref{thm:finite}. \\  

		\item Closed quadratic inequalities: All of the above results are for the case of open quadratic inequalities, whereas typically in mathematical programming we are interested in sets defined by closed quadratic inequalities. The situation with closed inequalities is much more delicate, as we illustrate in Example \ref{example:closed}. Much of the difficulty comes from the fact that sets defined by closed inequalities can have low-dimensional connected components. Previously known results make topological assumptions to avoid this situation. 
		In particular, it was shown in~\cite{ModaresiV2015} that if the set defined by closed inequalities has no lower-dimensional connected components, then the closures of convex hulls of sets defined by closed or open inequalities are the same; this allows us to transfer results from the open case to the closed case, under a topological assumption. Unfortunately, this assumption is hard to check computationally, and it does not hold in some interesting cases. 
		We show that if hidden hyperplane convexity holds, and the zero matrix is not a non-trivial aggregation, then the interior of the convex hull of the set given by closed quadratic inequalities is equal to the interior of the intersection of aggregations (Theorem \ref{prop:agg_interior}). While these are restrictive conditions, they do not make topological assumptions on the set defined by closed inequalities. 
	\end{itemize}
	
	The rest of the papers is organized as follows: In Section~\ref{sec:2} we present all our main results. In particular, in Section~\ref{sec:pre} we establish notation and preliminary results followed by Section~\ref{sec:HHC0}-Section~\ref{sec:closed0} where all the results are stated and explained. Section~\ref{sec:conclusion} presents conclusions and open questions. 
	Proofs of the results presented in Section~\ref{sec:2} are given in Sections \ref{sec:HHC}-\ref{sec:closed}.
	
	\section{Main Results}\label{sec:2}
	\subsection{Notations and preliminaries}\label{sec:pre}
	
	


	Given a positive integer $n$, we let $[n]$ denote the set $\{1, \dots, n\}$. Given a set $U \subseteq \mathbb{R}^n$, we use $\textup{dim}(U)$, $\conv(U)$, $\textup{int}(U)$, $\overline{U}$, and $\partial U$ to represent the dimension of $U$, the convex hull of $U$, the standard topological interior of $U$, the standard topological closure of $U$, and the boundary of the set $U$ respectively. Given a linear subspace $L$ of $\mathbb{R}^n$, we denote its orthogonal complement by $L^{\perp}$. For a square matrix $M$, we use $\textup{det}(M)$ to denote the determinant of $M$. We use $I_n$ to denote the $n\times n$ identity matrix and $e_i$ for the  $i$-th standard basis vector.
	
	Our main goal is to study sets defined by multiple open quadratic constraints:
	\begin{eqnarray}\label{eq:sdefn}
		S:= \{ x \in \mathbb{R}^n: x^{\top}A_i x +2b_i^{\top}x + c_i < 0 , i \in [m] \},
	\end{eqnarray}
	where $m \geq 2$ and $n \geq 3$. We also use the following notation:
	
	\begin{enumerate}
		\item Let $f_i(x)=x^{\top}A_i x +2b_i^{\top}x + c_i $ be quadratic functions defining (\ref{eq:sdefn}) and $Q_i=\begin{bmatrix}
			A_i & b_i\\b_i^{\top} & c_i
		\end{bmatrix}$ the corresponding matrices. We define \emph{homogenization} $f_i^h$ of $f_i$ to be the quadratic form given by $f^h_i(x,x_{n+1})=(x,x_{n+1})^\top Q_i (x,x_{n+1})$. \\
		
		\item The homogenized set $S^h$: $$S^h := \left\{(x, x_{n +1}) \in \mathbb{R}^{n}\times \mathbb{R}^{1}: x^{\top}A_ix + 2(b_i^{\top}x)x_{n+1}+ c_i x_{n+1}^2 < 0, i \in [m]\right\}.$$
		\smallskip

		\item The aggregation of constraints $S_{\lambda}$ and its homogenization $(S_{\lambda})^h$. For $\lambda \in \mathbb{R}^m_+$, we let $$Q_\lambda=\sum_{i=1}^m \lambda_iQ_i \hspace{.4cm} \text{and} \hspace{.4cm} F_\lambda=\sum_{i=1}^m \lambda_if_i$$ be the aggregated matrices and quadratic functions. Additionally we define:

		\begin{equation*}
			\begin{aligned}
				S_{\lambda} := \{ x \in \mathbb{R}^n  &  :  F_\lambda < 0 \}, \\
				(S_{\lambda})^h  := \{ (x,x_{n+1}) \in \mathbb{R}^{n+1} & :  F_\lambda^h(x,x_{n+1}) < 0 \}.
			\end{aligned}
		\end{equation*}
		Observe that $S\subseteq S_{\lambda},S^h\subseteq (S_{\lambda})^h$ for any nonzero $\lambda \in\R^m_+$. \\ 

		\item Let $$\Omega = \{\lambda \in \mathbb{R}^m_+\setminus\{0\} : \conv(S)\subseteq S_{\lambda} \textup{ and }Q_{\lambda} \textup{ has at most one negative eigenvalue.}\}$$
		
		Informally, $\Omega$ is the set of ``good" aggregations where $S_{\lambda}$ consists of one or two convex connected components, and $\conv(S)$ lies entirely in one of them. We will formally state and prove this equivalence in Lemma~\ref{lem:Omega} in Section~\ref{sec:more_quadratics}. \\ 
		
		\item Positive definite linear combination (PDLC): Given a set of  symmetric matrices $Q_1$, \dots $Q_m$, we say they satisfy PDLC if $\sum_{i = 1}^m \theta_i Q_i\succ 0$ holds for some $\theta\in\R^m$. \\
	\end{enumerate}
	
	The cases of $m=2$ and $m=3$ are studied in \cite{yildiran2009convex} and \cite{dey2021obtaining} respectively. If $S\ne\emptyset$ and $\conv(S)\ne\R^n$, then in the case of two quadratic inequalities the convex hull is always given by aggregations in $\Omega$, and in the case of three quadratic inequalities we need the additional PDLC condition. Notice that in the case of two quadratic inequalities, by taking $Q_3=-I$ the latter result implies the former one\footnote{If any aggregation uses a non-zero weight on the quadratic constraint corresponding to $-I$, then we can obtain a tighter aggregated constraint by setting the weight on this constraint to zero.}. The author of \cite{yildiran2009convex} also proved that for two quadratic inequalities, aggregations also give certificates when $S=\emptyset$ or $\conv(S)=\R^n$.  Moreover, \cite{yildiran2009convex} also showed at most two good aggregations suffice to define the convex hull.

	
	

	
	
	
	\subsection{Hidden hyperplane convexity}\label{sec:HHC0}
	We call a map  $\varphi: \mathbb{R}^n \rightarrow \mathbb{R}^m$ a quadratic map, if there exist $m$ symmetric matrices $Q_1, \dots Q_m$ such that:
	$$\varphi(x) = \left(x^{\top}Q_1x, \dots, x^{\top}Q_m x\right) \ \textup{for all } x \in \mathbb{R}^n.$$ 
	A quadratic map $\varphi:\mathbb{R}^n \to \mathbb{R}^m$ satisfies \emph{hidden convexity} if $\textup{image} \left( \varphi \right) = \{\varphi(x):x\in \R^n\}\subseteq \mathbb{R}^m$ is convex. We say that $n\times n$ symmetric matrices $Q_1,\ldots,Q_m$ satisfy hidden convexity if the map $x\mapsto (x^{\top}Q_1 x,\ldots,x^{\top}Q_m x)$ satisfies hidden convexity. 
	
	
	We now introduce a new notion of \emph{hidden hyperplane convexity} of quadratic maps, which will be our key assumption in proving that the convex hull of a set defined by quadratic inequalities is given by aggregations of these inequalities. We say $H \subseteq \R^n$ is a linear hyperplane if $H$ is a linear subspace of $\R^n$ with $\textup{dim}(H) = n -1$.
	
	\begin{definition}[Hidden hyperplane convexity (HHC)]
		A quadratic map $\varphi:\mathbb{R}^n\to \mathbb{R}^m$ satisfies hidden hyperplane convexity (HHC) if for all linear hyperplanes $H\subseteq \mathbb{R}^n$, 
		$\textup{image} \left( \varphi|_H \right) = \{\varphi(x):x\in H\}\subseteq \mathbb{R}^m$ is a convex set. Let $Q_1,\ldots,Q_m$ be $n\times n$ symmetric matrices. We say $Q_1,\ldots,Q_m$ satisfy HHC if the map $x\mapsto (x^{\top}Q_1 x,\ldots,x^{\top}Q_m x)$ satisfies HHC. 
	\end{definition}
	
	We present some properties of hidden hyperplane convexity. 
	
	\blue{
		\begin{remark}
			If the matrices $Q_i$ are linearly independent, then we must have $n\geq m$ for hidden convexity, and $n-1 \geq m$ for hidden hyperplane convexity. For hidden convexity, suppose $n<m$, and consider the span of the image of the quadratic map $\varphi$. Since the image is convex, the span has the same dimension as the image, and it is at most $n$. This means for all $x\in \R^n$ we have there exist $\lambda_i$, not all zero, such that $\sum \lambda_ix^{\top}Q_i x=\sum x^{\top}\lambda_iQ_i x=0$, and therefore $Q_i$ are linearly dependent. Contradiction. 
			
			For hidden hyperplane convexity, suppose $n-1<m$, and consider the span of the image of $\varphi$ restricted to a hyperplane. For a general hyperplane $H$ the restrictions of $Q_i$ to $H$ will be linearly independent and then the argument is same as for hyperplane convexity.

			

		More generally, we must have $\dim (\operatorname{span} \{Q_1,\dots,Q_m\}) \leq n$ (resp. $n-1$) for hidden convexity (resp. hyperplane hidden convexity) to hold. The proofs are the same as the ones outlined above.
\end{remark}}
We now observe that hidden hyperplane convexity implies hidden convexity.



\begin{observation}
	Hidden hyperplane convexity implies the usual hidden convexity as long as $n\ge 3$. Given $x,y\in \mathbb{R}^n$ we may pick some hyperplane $H$ containing both $x$ and $y$, and the segment between $\varphi(x)$ and $\varphi(y)$ in $\mathbb{R}^m$ is then contained in $\textup{image}(\varphi|_H)\subseteq \im \varphi$.
\end{observation}
On the other hand, HHC is a strictly stronger condition than hidden convexity as the next example illustrates.


\begin{example}[Hidden convexity does not imply hidden hyperplane convexity]\label{example:separable}
	Consider a diagonal quadratic map $\varphi: \mathbb{R}^n \rightarrow \mathbb{R}^m$, 
	$\varphi(x) = \left(x^{\top}D_1x, \dots, x^{\top}D_m x\right),$
	where $D_1, \dots, D_m$ are diagonal matrices; such a map is also sometimes referred to as a separable quadratic map. Any diagonal quadratic map $\varphi$ is known to satisfy hidden convexity (see Proposition 3.7 in \cite{polyak1998}), and we include a quick proof here. Given $x,y \in \mathbb{R}^n$ and $\lambda \in [0, 1]$, let $z \in \mathbb{R}^n$ be defined as 
	$$z_j = \sqrt{\lambda x_j^2 + (1 - \lambda)y_j^2}, \ j \in [n].$$ Then it is straightforward to verify that
	$$\lambda\varphi(x) + (1 - \lambda)\varphi(y) = \varphi(z),$$
	that is, $\textup{image} \left( \varphi \right)$ is convex.
	
	On the other hand, we show that a diagonal quadratic map may not satisfy hidden hyperplane convexity: 
	Let $\varphi: \mathbb{R}^4 \rightarrow \mathbb{R}^3$ be given by $f_1=x_1^2$, $f_2=x_2^2$ and $f_3=x_3^2$. The image of $\mathbb{R}^4$ is the non-negative orthant in $\mathbb{R}^3$. Now let's consider restrictions of $\varphi$ to a \blue{linear} hyperplane $H$. \blue{It is clear that in this specific example, $\{\varphi(x):x\in H\}=\{\varphi(x):x\in \pi(H) \}$, where $\pi:\R^4\to \R^3,\pi(x_1,x_2,x_3,x_4)=(x_1,x_2,x_3,0)$ is the  linear projection that forgets the last coordinate.} If $H$ does not contain the vector $(0,0,0,1)$, \blue{then $\pi(H)=\R^3$, and $\{\varphi(x):x\in H\}$} is the non-negative orthant in $\mathbb{R}^3$, and hidden hyperplane convexity on $H$ holds. If $H$ does contain $(0,0,0,1)$, then \blue{the image of $H$ under $\varphi$ may not be convex. For instance, let $H=\textup{span}\{(1,1,0,0),(1,0,1,0),(0,0,0,1)\}=\{(s+t,s,t,u)\subseteq \R^4: s,t,u\in\R\}$. Then $A=\textup{image }\varphi|_H=\{((s+t)^2,s^2,t^2)\subseteq \R^3: s,t\in\R \}$ is not convex, since $(4,1,1),(0,1,1)\in A$ but $(2,1,1)\notin A$. }
	Thus we see that $\varphi$ satisfies hidden convexity, but does not satisfy hidden hyperplane convexity. \blue{This example is interesting in the sense that for a dense subset of linear hyperplanes, the image of this quadratic map restricted to these hyperplanes is convex, and yet this convexity does not hold for all linear hyperplanes.}

\end{example}

We now show that hidden hyperplane convexity is preserved under the following two different operations. 

\begin{lemma}\label{lem:HHC_preserved}
	Suppose that $Q_1,\ldots, Q_m$ satisfy HHC. Then the following matrices also satisfy HHC: 
	
	\begin{enumerate}
		\item $P^{\top}Q_1 P,\ldots,P^{\top}Q_m P$ where $P$ is any invertible matrix. 
		
		\item $Q'_1,\ldots,Q'_k$ where $\span(Q'_1,\ldots,Q'_k)\subseteq \span(Q_1,\ldots,Q_m)$. (Equivalently, there exists a $k\times m$ matrix $\Lambda$ such that $Q'_i=\sum_{j=1}^m \Lambda_{ij}Q_j$ for all $i\in[k]$.)
	\end{enumerate}
\end{lemma}

\begin{proof}
	Let $H$ be any hyperplane in $\R^n$. For the first statement, we have $$U:= \{(x^{\top}P^{\top}Q_1 Px,\ldots,x^{\top}P^{\top}Q_m Px):x\in H\}=\\ \{(x^{\top}Q_1 x,\ldots,x^{\top}Q_m x):x\in H'\},$$ where $H'=\blue{\{Px:x\in H\}=PH}$ is also a hyperplane in $\R^n$. Thus, by HHC of $Q_1,\ldots, Q_m$ the set $U$ is also convex. 
	
	For the second statement, since $\span(Q'_1,\ldots,Q'_k)\subseteq \span(Q_1,\ldots,Q_m)$, there exists a $k\times m$ matrix $\Lambda$ such that $Q'_i=\sum_{j=1}^m \Lambda_{ij}Q_j$ for all $i\in[k]$. Then we have $\{(x^{\top}Q'_1 x,\ldots,x^{\top}Q'_k x):x\in H\}=\Lambda\{(x^{\top}Q_1 x,\ldots,x^{\top}Q_m x):x\in H\}$, which is convex since convexity is preserved under linear transformations. 
\end{proof}
The above result is important, especially (2), since it shows that hidden \blue{hyperplane} convexity is a property of linear subspaces of the space of symmetric matrices rather than a property that holds for some arbitrary subset of quadratic maps.

Our next observation is that hidden hyperplane convexity can be formulated with matrices. Let $\varphi:\mathbb{R}^n\to \mathbb{R}^m$ be a quadratic map. Let $H$ be any hyperplane of $\R^n$ and the columns of the matrix $W_H \in \mathbb{R}^{n\times (n-1)}$ be any basis for $H$. Then note that:
\begin{eqnarray*}
	\textup{image}(\varphi|_H) &=& \{ (x^{\top}Q_1x, \dots, x^{\top}Q_mx\blue{)}: x\in H\} \\
	&=& \{ (y^{\top}W_H^{\top}Q_1 W_Hy, \dots, y^{\top}W_H^{\top}Q_m W_Hy): y\in \blue{\mathbb{R}^{n-1}}\} \\
	&=& \textup{image}(\varphi(H)),
\end{eqnarray*}
where $\varphi(H): \blue{\mathbb{R}^{n-1} \rightarrow \mathbb{R}^m}$ is the \blue{quadratic} map: $y \rightarrow  (y^{\top}W_H^{\top}\blue{Q_1} W_Hy, \dots, y^{\top}W_H^{\top}\blue{Q_m} W_Hy)$. On the other hand, the columns of any full rank $n\times (n-1)$ matrix form a basis for some linear hyperplane. Thus, we arrive at the following equivalence.
\begin{observation}\label{obs:red} $Q_1,\ldots,Q_m$ satisfy HHC if and only if for all full-rank matrix $W\in\R^{n\times (n-1)}$, \\$W^{\top}Q_1 W,\ldots,W^{\top}Q_m W$ satisfy hidden convexity. 
\end{observation}
We obtain the following corollary of the above observation, using the classical hidden convexity theorems of Dines and Calabi.

\begin{corollary}[$m=2$, or $m=3$ with PDLC implies HHC]\label{cor:m23}
	\begin{enumerate}
		\item Let $Q_1,Q_2$ be symmetric matrices of dimension $n\ge 2$. Then $Q_1,Q_2$ satisfy HHC. 
		
		\item Let $Q_1,Q_2,Q_3$ be symmetric matrices of dimension $n\ge 4$. If $Q_1,Q_2,Q_3$ satisfy PDLC, then $Q_1,Q_2,Q_3$ satisfy hidden hyperplane convexity. 
	\end{enumerate}
\end{corollary}
\begin{proof} 
	\begin{enumerate}
		\item By Observation~\ref{obs:red}, it is sufficient to show that $W^{\top}Q_1 W,W^{\top}Q_2 W$ satisfy hidden convexity for any full-rank matrix $W\in\R^{n\times (n-1)}$. This follows from the classic theorem of Dines~\cite{dines}. 
		
		\item By Observation~\ref{obs:red}, it is sufficient to show that $W^{\top}Q_1 W,W^{\top}Q_2 W,W^{\top}Q_3 W$ satisfy hidden convexity for any full-rank matrix $W\in\R^{n\times (n-1)}$. Since $Q_1,Q_2,Q_3$ satisfy PDLC, there exists $\theta\in\R^3$ such that $\sum_{i = 1}^3 \theta_i Q_i \succ 0$. This implies that
		$$ \sum_{i = 1}^3 \theta_i (W^{\top} Q_i W) \succeq 0.$$
		Moreover since $W$ is full-rank, we have that $W y = 0$ iff $y = 0$. Thus, $\sum_{i = 1}^3 \theta_i (W^{\top} Q_i W) \succ 0$, proving that $W^{\top}Q_1 W,W^{\top}Q_2 W,W^{\top}Q_3 W$ satisfy PDLC. Therefore they satisfy hidden convexity due to a theorem of Calabi~\cite{calabi1982linear}. 
	\end{enumerate}
\end{proof}

In the following theorem we show a non-trivial example of hidden hyperplane convexity with an arbitrary number of quadratic functions. This shows that, while hidden hyperplane convexity is a strong assumption, it can lead to interesting examples of sets defined by quadratic inequalities, where the convex hull is given by aggregations.

\begin{theorem}[Non-trivial example of HHC with more constraints]\label{thm:HHTnontriv}
	Fix integers $n>m+1,m\ge \blue{1}$. Let $\varphi = (f_0,\ldots, f_m)$ where $f_0,\ldots, f_m:\R^n \to \R$ are quadratic forms on $\R^n$ such that $f_0$ is positive definite, and there exists linear form $\ell:\R^n\to\R$ such that for all $1\le i\le m$, $f_i(x)=\ell(x)\ell_i(x)$ for some linear form $\ell_i:\R^n\to\R$. Then $\varphi: \R^n \rightarrow \mathbb{R}^m$ satisfies HHC.
\end{theorem}
A proof of Theorem~\ref{thm:HHTnontriv} is presented in Section~\ref{sec:HHC}.

\subsection{Hidden hyperplane convexity and obtaining convex hull from aggregations}\label{sec:convex0}

Let $f_1,\ldots,f_m$ be $m$ (inhomogeneous) quadratic functions $f_i(x)=x^{\top}A_ix+2b_i^{\top}x+c_i$, and let $\hf_i$ be their homogenizations $\hf_i(x,x_{n+1})=x^{\top}A_ix+2(b_i^{\top}x)x_{n+1}+c_ix_{n+1}^2 = (x,x_{n +1})^{\top} Q_i(x, x_{n +1}) $. We denote by $\hf=(\hf_1,\ldots,\hf_m)$ the associated homogeneous quadratic map from $\R^{n+1}$ to $\R^m$.  

Our main result of this section states that the convex hull of the set $S$ defined by $f_i$ is given by aggregations if the associated quadratic map has hidden hyperplane convexity.

	
	\begin{theorem}\label{thm:more_quadratics}
		Let $n \geq3$ and $f_i:\mathbb{R}^n \rightarrow \mathbb{R}$ be the functions $f_i(x)=x^{\top}A_ix+2b_i^{\top}x+c_i,i\in [m]$. Let $S=\{x\in\R^n:f_i(x)<0,i\in [m]\}$. Suppose that the associated quadratic map $\hf$ satisfies the hidden hyperplane convexity. If $S\ne\emptyset$ and $\conv(S)\ne\R^n$, then $$\conv(S)=\bigcap_{\lambda\in \Omega} S_{\lambda}.$$
	\end{theorem}
	
	Our proof follows the same road-map as the proof in \cite{dey2021obtaining} for the case of three quadratics satisfying PDLC. One of the main ingredients of their proof is the homogeneous S-lemma, and we prove a similar result using the hidden hyperplane convexity assumption. See Section~\ref{sec:more_quadratics} for a proof of Theorem~\ref{thm:more_quadratics}.
	
	
	Corollary~\ref{cor:m23} states that HHC is always satisfied if $m=2$ or $m=3$ with PDLC. Thus Theorem~\ref{thm:more_quadratics} together with Corollary~\ref{cor:m23} recovers the main results of~\cite{yildiran2009convex} and~\cite{dey2021obtaining}. In fact, by using Lemma~\ref{lem:HHC_preserved} it is straightforward to see that we can obtain the following slightly more general result than presented in~\cite{yildiran2009convex} and~\cite{dey2021obtaining}. {
	
	\begin{theorem}\label{thm:dep}
		Suppose that $Q_1, \dots, Q_m$ satisfy the following: \begin{itemize}
			\item There exists two indices $i_1, i_2 \in [m]$ such that $Q_1, \dots, Q_m$ belong to the span of $Q_{i_1}, Q_{i_2}$, or, 
			\item There exists three indices $i_1, i_2, i_3 \in [m]$ such that $Q_1, \dots, Q_m$ belong to the span of $Q_{i_1}, Q_{i_2}, Q_{i_3}$ and $Q_{i_1}, Q_{i_2}, Q_{i_3}$ satisfy PDLC. 
		\end{itemize} If $\emptyset \subsetneq \textup{conv}(S) \subsetneq \mathbb{R}^n$, then $\conv(S)$ is given by aggregations, i.e., 
		\blue{$\conv(S) = \bigcap_{\lambda \in \Omega}S_{\lambda}.$}
	\end{theorem}
	
	We next evaluate the ``tightness" of Theorem~\ref{thm:more_quadratics} vis-\'a-vis the hidden hyperplane convexity condition. First via the following example, we show that the weaker condition of hidden convexity is not sufficient for convex hull to be given by aggregations. 
	
	\begin{example}[Hidden convexity is not sufficient]\label{example:hidden_convexity_bad}
		This example is in part inspired by Example 3.4 in \cite{polyak1998}. Let $n\ge 3$ and consider the following three quadratic functions on $\R^n$: $$f_1(x)=x_1^2-x_2^2, \,\,\, f_2(x)=x_1x_2,\,\,\,f_3(x)=-1-(x_1^2-x_2^2)-x_1x_2+\sum_{i=3}^n x_i^2.$$ 
		It is straightforward to verify that $f_1,f_2,f_3$ satisfy hidden convexity, as the image of the associated homogeneous quadratic map is $\R^3$. Observe that PDLC does not hold as the coefficients of $x_1^2$ and $x_2^2$ either have different signs, or are both zero in any linear combination of $f_i$. 
		
		We now show that any good aggregation must be a scalar multiple of $f_1+f_2+f_3$, and the set defined by all good aggregations is $\{x:\sum_{i=3}^n x_i^2<1\}$, which has no restrictions on $x_1$ and $x_2$. Observe that $(-0.1,0.9,0,\ldots,0),(0.5,-0.7,0,\ldots,0)\in S$ and hence their midpoint  $x^*=(0.2,0.1,0,\ldots,0)$ lies in $\conv(S)$. Let $\lambda\ge 0$ be any good aggregation, i.e., $Q_{\lambda}$ has at most one negative eigenvalue and $\conv(S)\subseteq S_{\lambda}$. In particular $\sum_{i=1}^3\lambda_i f_i(x^*)<0$ as $x^*\in\conv(S)$. Since $f_1(x^*)>0,f_2(x^*)>0,f_3(x^*)<0$ we must have $\lambda_3>0$, \blue{which means the bottom right diagonal element of $Q_{\lambda}$ is negative.} Since $Q_{\lambda}$ has at most one negative eigenvalue \blue{and the last row and column are always zero except for the diagonal element (there are no linear terms), the leading $n\times n$ principal submatrix of $Q_{\lambda}$ must be PSD, which means} $\lambda_1=\lambda_2=\lambda_3$. 
		
		On the other hand, the actual convex hull is given by $\{x:x_2^2<1-\sum_{i=3}^n x_i^2,(2x_1+x_2)^2<1-\sum_{i=3}^n x_i^2\}$. Geometrically for any fixed $x_3,\ldots,x_n$ such that $\sum_{i=3}^n x_i^2<1$, the set of feasible $(x_1,x_2)$ lies inside an open parallelogram, with vertices $(-a,a),(0,a),(a,-a),(0,-a)$ where $a=\sqrt{1-\sum_{i=3}^n x_i^2}$. 
	\end{example}
	
	We next ask whether hidden hyperplane convexity is a necessary condition for obtaining the convex hull of a set defined by quadratic inequalities using aggregations. As shown in Example~\ref{example:separable} diagonal quadratic functions may not satisfy HHC. However, in the next result we show that if $S$ is defined by diagonal quadratic inequalities then $\textup{conv}(S)$ is always given by aggregations. See \cite{burer2020exact} for a study of semidefinite relaxations of related sets. 
	
	\begin{theorem}\label{thm:diagonal} (HHC not necessary; Separable quadratic maps) Let $n \geq 2$ and $f_i:\mathbb{R}^n \rightarrow \mathbb{R}$ be the functions $f_i(x)=x^{\top}A_ix+2b_i^{\top}x+c_i,i\in [m]$. Let $S=\{x\in\R^n:f_i(x)<0,i\in [m]\}$. Assume $Q_1, \dots, Q_m$ are diagonal. Then:
		\begin{enumerate}
			\item $S = \emptyset$ if and only if there exists nonzero $\lambda \geq 0$ such that $\sum_{i = 1}^m \lambda_i Q_i \succeq 0$.
			\item If $S\ne\emptyset$, then $\textup{conv}(S) \neq \mathbb{R}^n$ if and only if there exists $\lambda \geq 0$ such that the leading $n \times n $ principal submatrix of $\sum_{i = 1}^m \lambda_i Q_i$ is nonzero and positive semidefinite, i.e. there exists $\lambda \geq 0$ such that the set $\{x: \sum_{i = 1}^m \lambda_i f_i(x) < 0 \}$ is convex \blue{and not $\R^n$}. 
			\item If $\emptyset \subsetneq \textup{conv}(S) \subsetneq \mathbb{R}^n$, then $conv(S)$ is described by finitely many aggregations, where the leading $n \times n$ principal submatrix of each aggregation is positive semidefinite.
		\end{enumerate}
	\end{theorem}
	A proof of Theorem~\ref{thm:diagonal} is provided in Section~\ref{sec:diag}.
	
	In Theorem \ref{thm:more_quadratics} we assume $S\ne\emptyset$ and $\conv(S)\ne\R^n$, and it is natural to ask what happens if either assumption fails. We show that non-emptiness of $S$ can be certified using aggregations under the weaker assumption of hidden convexity (without requiring HHC). The situation for $\conv(S)\ne\R^n$ is more nuanced, and aggregation certificates suffice except for one case. 
	
	
	
	\begin{proposition}[Hidden convexity certifies non-emptiness of $S$]\label{prop:empty}
		Let $f_1,\ldots,f_m$ be quadratic functions where the image of the associated homogeneous quadratic map $\hf=(\hf_1,\ldots,\hf_m):\R^{n+1}\to\R^m$ is convex. Then $S=\emptyset$ if and only if $Q_{\lambda}\succeq 0$ for some nonzero $\lambda \in \mathbb{R}^m_{\geq 0}$. 
	\end{proposition}
	A proof of Proposition~\ref{prop:empty} is given in Section~\ref{sec:emptyrn}.

	
	
	
	
	Note that if there exists \blue{a nonzero} $\lambda\ge 0$ such that $\sum_{i = 1}^m \lambda_i A_i\succeq 0$ and furthermore $\sum_{i=1}^m \lambda_i f_i(x)$ is not a negative constant \blue{function}, then $\conv(S)\ne\R^n$. We now show the partial converse that if no nonzero $\lambda\ge 0$ satisfies $\sum_{i = 1}^m \lambda_i A_i\succeq 0$ then $\conv(S)=\R^n$ if we assume hidden convexity on a particular hyperplane. The unresolved case is where $\sum_{i=1}^m \lambda_i f_i(x)$ is a negative constant for all nonzero $\lambda\ge 0$ satisfying $\sum_{i = 1}^m \lambda_i A_i\succeq 0$. 
	
	\begin{proposition}\label{prop:Rn}
		Let $\mathcal{E}=\{(x,x_{n+1})\in\R^{n+1}:x_{n+1}=0 \}$ and assume $\im\hf|_{\mathcal{E}}$ is convex. Assume there does not exist nonzero $\lambda\in\R^m_+$ such that $\sum_{i =1}^m \lambda_i A_i\succeq 0$. Then $\conv(S)=\R^n$.  
	\end{proposition}
	
	It is clear that this assumption is weaker than hidden hyperplane convexity, which requires $\im\hf|_H$ to be convex for any hyperplane $H\subseteq\R^{n+1}$.  A proof of Proposition~\ref{prop:Rn} is presented in Section~\ref{sec:emptyrn}.

	\subsection{Convex hull of sets defined by linear and sphere constraints}
	

	
	In Theorem~\ref{thm:HHTnontriv} we prove hidden hyperplane convexity of a special class of quadratic maps. This results  leads to the following theorem on sets defined by linear 
	and sphere inequalities. 
	
	\begin{theorem}\label{thm:cheese}
		Let $f_i(x)=x^{\top}A_i x+2b_i^{\top}x+c_i,1\le i\le m$ be quadratic functions on $\R^n$, where $A_i$ is either $I_n$ for $i \in P \subseteq [m]$,$-I_n$ for $i \in N \subseteq [m]$ or $0$ for $i \in Z \subseteq [m]$. Let $S=\{x\in\R^n: f_i(x)<0,i\in [m]\}$. 
		
		
		Then:
		\begin{itemize}
			\item $S=\emptyset$ if and only there exists some nonzero $\lambda\ge 0$ such that $Q_{\lambda}\succeq 0$ (which can be checked using an SDP.)
			\item $\conv(S)\ne\R^n$ if and only if either $P\ne\emptyset$ or there exists $i\in Z$ such that $b_i\ne 0$ or $b_i=0$ and $c_i\ge 0$. 
			\item If $\emptyset \subsetneq \textup{conv}(S) \subsetneq \mathbb{R}^n$, and either $m\le n-1$ or $m\le n$ and PDLC condition holds, 
			then $\conv(S)$ can be described by at most $|P||N|+|P|+|Z|$ aggregations.
		\end{itemize}
	\end{theorem}
	
	A proof of Theorem~\ref{thm:cheese} is presented in Section~\ref{sec:cheese}. Here is an example where $m=n=3$ and PDLC holds. 
	
	\begin{example}\label{example:3spheres}
		Consider the following three quadratics in $\R^3$. 
		
		$$f_1(x)=x_1^2 + x_2^2 + x_3^2-2x_3-1 ,\,\, f_2(x)=x_1^2 + x_2^2 + x_3^2+2x_3-4  ,\,\,  f_3(x)=-x_1^2-x_2^2-x_3^2+1$$
		
		PDLC is satisfied with $\theta=(-1,-1,-3)$, and in this case $|P|=2,|N|=1,|Z|=0$. Thus by Theorem~\ref{thm:cheese} the convex hull is given by at most four aggregations. In fact, three aggregations suffice for this example. A plot of this region is given in Figure \ref{fig:3sph}.
	\end{example}
	
	\begin{figure}[h] 
		\centering
		\includegraphics[scale=0.75]{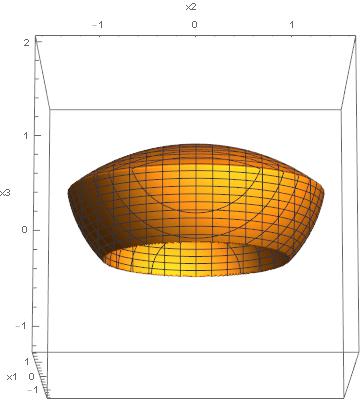}
		\caption{Plot of Example~\ref{example:3spheres}. }
		\label{fig:3sph}
	\end{figure}
	
	Note each $f_i$ defines a region that is either a ball if $A_i=I$, a linear halfspace if $A_i=0$, or the complement of a ball if $A_i=-I$. Thus Theorem~\ref{thm:cheese} applies to sets defined by linear and sphere inequalities. Such sets appear in the context of trust region subproblems, and have been studied in papers  such as~\cite{yang2018quadratic,bienstock2014polynomial,ho2017second}. We make no assumptions on the constraints, as opposed to \cite{yang2018quadratic} which requires the outside-the-ball constraints ($A = -I$) to be non-intersecting. On the other hand we only study $\conv(S)$ instead of the convex hull in lifted space $\{(x,xx^{\top}):x\in S \}$ or algorithms to solve the trust region problem as in~\cite{bienstock2014polynomial}. If $|P|+ |N| = 1$, then the set in Theorem~\ref{thm:cheese} is related to the set studied in~\cite{asterquad}. \blue{We also note that in this case every good aggregation is in fact convex, i.e., each good aggregation defines a single convex set, and the set of good aggregations $\Omega$ is polyhedral. Therefore the results in \cite{wang2022tightness} apply to this case.}

	The following example shows that convex hull may not be given by aggregations when there are $n+1$ (linearly independent) linear and sphere constraints and PDLC does not hold. We do not know whether an example with $n$ linearly independent constraints (without PDLC) exists. 
	
	\begin{example}\label{example:sphere_n+1}
		Let $$f_0=1-\sum_{i=1}^n x_i^2 \hspace{.5cm} \text{and} \hspace{.5cm} f_i=-x_i,i\in[n].$$ We claim that $\conv(S)=\{x>0:\sum_{i=1}^n x_i>1 \}$. It is clear $\varepsilon\bar{1}+\alpha e_i\in S$ for all $\varepsilon>0,\alpha>1,i\in [n]$, and therefore $\{x>0:\sum_{i=1}^n x>1 \}\subseteq\conv(S)$. To show $\conv(S)\subseteq\{x>0:\sum_{i=1}^n x>1 \}$, it suffices to show $\sum_{i=1}^n x_i>1 $ holds for all $x\in S$. Suppose this does not hold, then there exists $y>0$ with $\sum_{i=1}^n y_i\le 1$, then $0\le y_i\le 1$ for all $i$, and $\sum_{i=1}^n y_i^2\le \sum_{i=1}^n y_i\le 1$, which means $f_0(y)\ge 0$ and $y\notin S$. 
		
		We now show that the set defined by all good aggregations is the positive orthant, which is different from the actual convex hull. Let $\lambda\ge 0$ be any good aggregation. Since $Q_{\lambda}$ has at most one negative eigenvalue, we must have $\lambda_0=0$. Thus all good aggregations are nonnegative linear combinations of $f_i$.
	\end{example}



	
	
	
	\subsection{Together hidden hyperplane convexity and PDLC for every triple lead to finite number of good aggregations defining the convex hull}  
	\label{sec:finite0}

	
	In the case where convex hull can be described by aggregations, a natural question is whether finitely many aggregations suffice. When $m = 2$, this is already shown to be true~\cite{yildiran2009convex}. Note that we also verified this for the special cases as stated in Theorem~\ref{thm:cheese} and Theorem~\ref{thm:diagonal}. 
	
	
	This question was raised in \cite{dey2021obtaining} for three quadratics under PDLC condition. Here we give an affirmative answer for three quadratics, and consider the question in the more general setting of $m$ quadratics under hidden hyperplane convexity assumption, where every triple of quadratics satisfies PDLC condition. 
	
	\begin{theorem}\label{thm:finite}
		Let $n \geq3$ and $f_i:\mathbb{R}^n \rightarrow \mathbb{R}$ be the functions $f_i(x)=x^{\top}A_ix+2b_i^{\top}x+c_i,i\in [m]$. Let $S=\{x\in\R^n:f_i(x)<0,i\in [m]\}$. Assume $S \neq \emptyset$ and $\textup{conv}(S) \neq \mathbb{R}^n$ and HHC holds for the associated homogeneous quadratic map $\hf$, so that Theorem \ref{thm:more_quadratics} holds and $$\conv(S)=\bigcap_{\lambda\in\Omega_1}S_{\lambda},$$ where $\Omega_1=\Omega\setminus\{\lambda\in\Omega:S_{\lambda}=\R^n \}$. 
		Furthermore assume for all distinct $i,j,k\in [m]$ there exist scalars $p_{ijk},q_{ijk},r_{ijk}\in\R$ such that $p_{ijk} Q_i+q_{ijk}Q_j+r_{ijk}Q_k\succ 0$. Then there exist $\lambda^{(1)},\ldots,\lambda^{(r)}\in\Omega_2$ such that $$\conv(S)=\bigcap_{i=1}^r S_{\lambda^{(i)}},$$ where  $\Omega_2=\{\lambda\in\Omega_1: |\{i:\lambda_i>0\}|\le 2 \}$ and $r \leq m^2 - m$. 
		
		Moreover, given any $u, v \in [m]$, $u \neq v$, there are at most two $\lambda^{(i)}$s with support $u$, $v$. Furthermore, these $\lambda^{(i)}$s can be written as $\alpha'e_u+(1-\alpha')e_v,\alpha''e_u+(1-\alpha'')e_v$, where $\alpha',\alpha''$ are roots of $\det (\alpha Q_u + (1-\alpha)Q_v)=0$. 
	\end{theorem}
	A proof of Theorem~\ref{thm:finite} is presented in Section~\ref{sec:finite}.
	The key ideas to prove Theorem~\ref{thm:finite} are the following:
	\begin{itemize}
		\item Given an aggregation $S_{\lambda}$, one can obtain an improved aggregation $S_{\tilde{\lambda}}$, i.e. $S_{\tilde{\lambda}} \subseteq S_{\lambda}$, such that $\tilde{\lambda} \in \Omega_1$. This is obtained as $\tilde{\lambda} = \lambda + \theta$ where $Q_{\theta}\succeq 0$. (Proposition~\ref{prop:improve})
		\item The idea is to repeatedly improve along such positive definite linear combinations so as to reduce the support of aggregations that are required to obtain the convex hull to at most $2$. (Proposition~\ref{prop:2aggregation})
		\item Now among aggregations that have support of fixed two indices, say $i$ and $j$, it is shown that at most two aggregations are sufficient. (Proposition~\ref{prop:pairwise_endpoints})
	\end{itemize}
	
	\begin{remark}
		As discussed in \cite{dey2019convexifications}, the closure of each component defined by a good aggregation  is second-order cone representable (SOCr). Thus finiteness of good aggregations implies that the closure of convex hull is SOCr, since it is given by intersection of finitely many components which are all SOCr. 
	\end{remark}
	
	For the case of $m =3$, note that PDLC implies hidden hyperplane convexity, so PDLC is sufficient to guarantee that no more than $3^2  - 3 = 6$ aggregations are sufficient to obtain the convex hull, answering a question raised in \cite{dey2021obtaining}. 
	
	\begin{corollary}\label{cor:6_agg}
		Let $f_1,f_2,f_3$ be three quadratic functions such that there exist $\theta\in\R^3$ such that $\sum_{i=1}^3 \theta_i Q_i\succ 0$. Suppose $S\ne\emptyset$ and $\conv(S)\ne\R^n$. Then there exists $\Omega'\subseteq \Omega,|\Omega|\le 6$ such that $\conv(S)=\bigcap_{\lambda\in\Omega'}S_{\lambda}$. 
	\end{corollary}
	The following example for $m=3$ case requires $4$ aggregations to describe the convex hull.
	
	\begin{example}\label{example:4}
		Consider the set $S$ described by the following functions:
		\begin{equation*} 
			\begin{aligned}
				f_1(x)&=-x_1^2+1+\sum_{i=2}^n x_i^2\\
				f_2(x)&=x_1^2+5x_1-4+\sum_{i=2}^n x_i^2\\
				f_3(x)&=-x_1-\sum_{i=2}^n x_i^2
			\end{aligned}
		\end{equation*}
		so that $S=\{x\in\R^n:f_i(x)<0,i\in[3]\}$. We have $-7f_1(x)-3f_2(x)-15f_3(x)=4x_1^2+5\sum_{i=2}^n x_i^2+5$, corresponding to a positive definite matrix. Thus PDLC holds and therefore $\conv(S)$ is given by at most 6 aggregations. In fact, $\conv(S)$ is described by 4 aggregations: $$\conv(S)=\{x\in\R^n:f_1(x)<0,f_2(x)<0,f_1(x)+f_3(x)<0,f_2(x)+f_3(x)<0\}.$$ 
		
		\blue{A plot of Example \ref{example:4} when $n=2$ is given by Figure \ref{fig:4agg}. The set $S$ is represented by the black shaded region, which has two connected components. The two aggregations $f_1(x)+f_3(x)<0,f_2(x)+f_3(x)<0$ give us the two vertical lines that join the left and right tips of both components. }
		
	\end{example}
	
	\begin{figure}[h]
		\centering
		\includegraphics[scale=0.75]{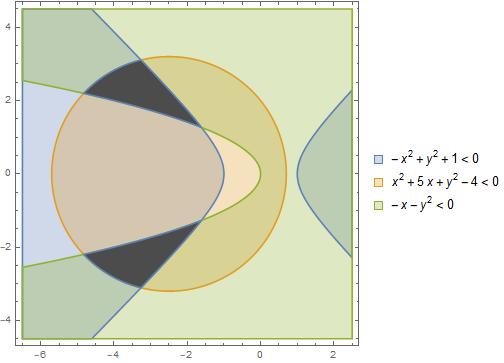}
		\caption{Plot of Example~\ref{example:4} with $n=2$. The black region represents $S$. }
		\label{fig:4agg}
	\end{figure}
	
	The bound of $m^2  - m$ can be improved in the special case where the quadratics $Q_i$ that defined $S$ span a linear space of dimension at most three. We have already shown in Theorem~\ref{thm:dep} that PDLC (or less, when the dimension of the span of the associated quadratic map is $2$) 
	is sufficient for the convex hull to be given by aggregations. We separate our discussion into two cases based on the dimension of the span of $Q_1,\dots,Q_m$.
	
	\paragraph{\emph{Span of $Q_1,\ldots,Q_m$ is two dimensional}}	
	
	Consider the conical hull of $Q_1,\ldots,Q_m$, which is a closed polyhedral cone of dimension two. If this cone is a linear subspace or a linear halfspace, i.e., there exists nonzero $\lambda\ge 0$ where $\sum_{i=1}^m \lambda_i Q_i=0$, then $S=\emptyset$. If $S\ne\emptyset$, then the conical hull of $Q_1,\ldots,Q_m$ is a closed pointed polyhedral cone of dimension two, which has exactly two extreme rays. Say the two extreme rays are generated by $Q_i,Q_j$ respectively, then all other $Q_k$ can be written as nonnegative combinations of $Q_i$ and $Q_j$. In this case $\conv(S)=\conv(S_i\cap S_j)$, reducing to the two quadratics case which is described in Theorem 1 of \cite{yildiran2009convex}. 
	
	\paragraph{\emph{Span of $Q_1,\ldots,Q_m$ is three dimensional}}	
	We make the following observations: if $S\ne\emptyset$, then the conical hull of $Q_1,\ldots, Q_m$, denoted by $\mathcal{C}$, must be a pointed closed polyhedral cone of dimension three, and we may assume without loss of generality that every $Q_i$ spans an extreme ray of $\mathcal{C}$. Then $\mathcal{C}$ has exactly $m$ facets, and each facet is generated by exactly two $Q_i$'s. 
	Given an aggregation that lies in the intersection of $\Omega_1$ and interior of $\mathcal{C}$, by Proposition \ref{prop:improve}, we may improve it using a positive definite combination until it touches the boundary of $\mathcal{C}$. Proposition \ref{prop:pairwise_endpoints} implies that intersection of all aggregations in $\Omega$ on a facet can be described by two endpoints. Since there are exactly $m$ facets, it follows that $2m$ aggregations suffice to describe the convex hull for the case of dependent quadratics. Therefore we obtain the following result.
	
	\begin{proposition}\label{prop:2m_agg}
		Let $n\ge 3$ and $f_1,\ldots,f_m$ be quadratic functions such that span of $Q_1,\ldots,Q_m$ is three dimensional, and there exists $\theta\in\R^m$ such that $\sum_{i=1}^m \theta_i Q_i\succ 0$. Suppose $S\ne\emptyset$ and $\conv(S)\ne\R^n$, then there exists $\Omega'\subseteq \Omega,|\Omega|\le 2m$ such that $\conv(S)=\bigcap_{\lambda\in\Omega'}S_{\lambda}$. 
	\end{proposition}
	
	\subsection{Closed Inequalities}\label{sec:closed0}
	Given quadratic functions $f_1,\ldots,f_m$, let $S=\{x:f_i(x)<0,i\in [m]\}$ be the set defined by open inequalities, and let $T=\{x:f_i(x)\le 0,i\in [m]\}$ be the one with closed inequalities. As usual let $Q_i=\begin{bmatrix}
		A_i & b_i\\ b_i^{\top} & c_i
	\end{bmatrix}$. Note $\conv(S)$ is always open but $\conv(T)$ may not be closed. In this section we study $G$, the interior of $\overline{\conv(T)}$. It is clear $G$ is convex, open and $\conv(S)\subseteq G$. 
	
	In \cite{dey2021obtaining} it was shown that when $G=\conv(S)$ and $\conv(S)$ is given by aggregations, then $\overline{\conv(T)}$ is given by the same aggregations after changing all open inequalities to closed. The original proof is only for the case of three quadratics with PDLC, but the same proof works for arbitrary number of quadratics with hidden hyperplane convexity.
	
	We do not make the assumption that $G =\conv(S)$, and show that \blue{$G$} is still given by aggregations of open inequalities, under \blue{HHC and} an additional technical assumption. 
	We now give an example where $G \not = \textup{conv}(S)$, which illustrates the delicate nature of closed inequalities. 
	
	\begin{example}[$G \not = \textup{conv}(S)$]\label{example:closed}
		Let $n\ge 2$ and consider the following two quadratic functions on $\R^n$:
		\begin{equation*}
			\begin{aligned}
				f_1(x)&=-x_1^2+x_1 \\
				f_2(x)&=-1+\sum_{i=1}^n x_i^2    
			\end{aligned}
		\end{equation*}
		Then $\conv(S)=S=\{x:f_1(x)<0,f_2(x)<0\}=\{x:x_1<0,\|x\|_2<1\}$. Note $T=\overline{S}\cup\{e_1\}$, and $G$ strictly contains $\conv(S)$. It turns out that $G$ is also given by aggregations $G=\{x:f_2(x) < 0,2f_1(x)+f_2(x)<0\}$, which is different from the ones defining $\conv(S)$. 
		
		\blue{A plot of Example \ref{example:closed} when $n=2$ is given in Figure \ref{fig:closed}. In the dimension 2 case, the set $G$ needs to contain the open triangle with vertices $(1,0),(0,1),(0,-1)$. } 
	\end{example}
	
	\begin{figure}[h]
		\centering
		\includegraphics[scale=0.75]{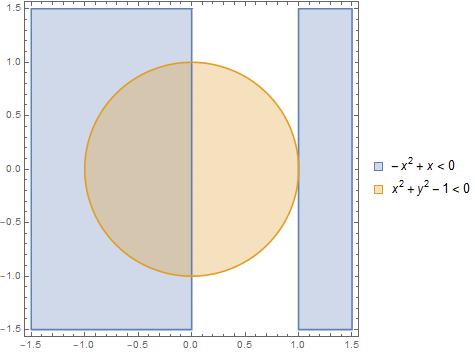}
		\caption{Plot of Example~\ref{example:closed} with $n=2$. }
		\label{fig:closed}
	\end{figure}
	
	We now state our main theorem for closed inequalities.
	
	\begin{theorem}\label{prop:agg_interior}
		Given quadratic functions $f_1,\ldots,f_m$, let $T=\{x:f_i(x)\le 0,i\in [m]\}$ and $G=\textup{int}(\overline{\conv(T)})$. Assume $Q_1,\ldots,Q_m$ satisfy hidden hyperplane convexity $\emptyset\subsetneq G\subsetneq\R^n$, and furthermore $Q_{\lambda}\ne 0$ for all nonzero $\lambda\ge 0$. Then $G=\bigcap_{\lambda\in\Omega_T}S_{\lambda}$, where $S_{\lambda}=\{x:\sum_{i=1}^m\lambda_if_i(x)<0\}$ and $\Omega_T\subseteq \R_+^m\setminus \{0\}$ is the set of $\lambda$ where $Q_{\lambda}=\sum_{i=1}^m\lambda_i Q_i$ has at most one negative eigenvalue and $G\subseteq S_{\lambda}$. 
	\end{theorem}

	\begin{remark}
		The condition that $Q_{\lambda}\ne 0$ for all nonzero $\lambda\ge 0$ is needed in our proof, and it is easy to check computationally using linear programming. This condition is satisfied if we assume PDLC and $S\ne\emptyset$. It allows for situations similar to Example~\ref{example:closed}. 
	\end{remark}
	
	A proof of Proposition~\ref{prop:agg_interior} is presented in Section~\ref{sec:closed}.
	
	\section{Conclusions and open questions}\label{sec:conclusion}

	We showed 
	that for a set described by any number of quadratics, hidden hyperplane convexity is a sufficient condition for convex hull to be given by good aggregations, assuming $S\ne\emptyset$ and $\conv(S)\blue{\ne}\R^n$. Furthermore Theorem \ref{thm:diagonal} and Example \ref{example:hidden_convexity_bad} together show that hidden hyperplane convexity is not necessary while hidden convexity is not sufficient. 
	
	We conjecture that even with hidden hyperplane convexity there exist sets $S$ defined by quadratic inequalities where infinitely many good aggregations are needed to define the convex hull.
	
	\begin{conjecture}\label{conj:4_infinite}
		There exists a set $S$ described by quadratic inequalities satisfying hidden hyperplane convexity, such that $\textup{conv}(S)$ cannot be described by finitely many good aggregations. 
	\end{conjecture}
	
	In Theorem \ref{thm:finite} we showed that $m^2-m$ good aggregations describe the convex hull under hidden hyperplane convexity and every triple PDLC assumption. In particular when $m =3$ and PDLC holds, six aggregations suffice, and we gave an example where four good aggregations are needed. We have not discovered an example which requires more than four good aggregations to describe the convex hull, but we conjecture that such examples exist. 
	\begin{conjecture}\label{conj:6}
		There exists a set $S$ described by three quadratic inequalities satisfying PDLC, such that $\textup{conv}(S)$ is described by using exactly six aggregations. 
	\end{conjecture}

	We showed that aggregations always certify when $S=\emptyset$, and almost always certify when $\conv(S)=\R^n$, except for one unresolved case, where $\sum_{i = 1}^m \lambda_i f_i$ is a negative constant for all nonzero $\lambda\ge 0$ such that $\sum_{i = 1}^m \lambda_i A_i\succeq 0$. We conjecture that under hidden hyperplane convexity assumption, in this case $\conv(S)$ is also $\R^n$. In other words we have the following conjecture, which states aggregations always provide certificates when $\conv(S)=\R^n$: 
	
	\begin{conjecture}\label{conj:Rn_general}
		Let $f_i(x)=x^{\top}A_ix+2b_i^{\top}x+c_i,i\in [m]$ be quadratic functions where the associated homogeneous quadratic map $\hf=(\hf_1,\ldots,\hf_m):\R^{n+1}\to\R^m$ has hidden hyperplane convexity. Assume $S=\{x:f_i(x)<0,x\in [m] \}$ is nonempty. Then $\conv(S)\ne\R^n$ if and only if there exists nonzero $\lambda\ge 0$ such that $\sum_{i =1}^m \lambda_i A_i\succeq 0$, and $\sum_{i=1}^m \lambda_i f_i$ is not a negative constant.  
	\end{conjecture}

	
	
	
	\section{HHC for a special class of quadratic forms}\label{sec:HHC}
	
	Our goal in this section is to prove that a special class of quadratic maps satisfies hidden hyperplane convexity. We first study the image of one particular quadratic map. 
	
	\begin{proposition}\label{prop:bd_soc}
		Let $n\ge 2$ and $f_0(x)=\sum_{i=1}^n x_i^2$, $f_i(x)=x_1x_i,1\le i\le n$ be $(n+1)$ quadratic forms on $\R^n$. Then the image of corresponding quadratic map is given by $$\{(f_0(x),\ldots,f_n(x)):x\in\R^n\}=\{(y_0,\ldots,y_n):y_0y_1=\sum_{i=1}^n y_i^2,y_0\ge 0,y_1\ge 0\},$$ which is linearly isomorphic to the boundary of second-order cone (Lorentz cone) in $\R^{n+1}$, given by $\{(z_0,\ldots,z_n):z_0^2=\sum_{i=1}^n z_i^2,z_0\ge 0 
		\}$. 
	\end{proposition}
	
	\begin{proof}
		``$\subseteq$'': This is clear since $f_0(x)f_1(x)=\sum_{i=1}^n f_i(x)^2$ and $f_0,f_1$ are sum of squares. 
		
		``$\supseteq$'': Given $(y_0,\ldots,y_n)\in\R^{n+1}$ satisfying $y_0y_1=\sum_{i=1}^n y_i^2,y_0\ge 0,y_1\ge 0$, we construct $x\in\R^n$ with $y_i=f_i(x)$. We divide this into two cases based on whether $y_1$ is zero or positive: 
		
		\begin{itemize}
			\item If $y_1=0$, then $\sum_{i=1}^n y_i^2=y_0y_1=0$ implies $y_i=0$ for all $1\le i\le m$. In this case we let $x_2=\sqrt{y_0}$ and $x_i=0$ otherwise. 
			
			\item If $y_1>0$, we let $x_1=\sqrt{y_1}$ and $x_i=\frac{y_i}{\sqrt{y_1}}$ for all $2\le i\le n$. 
		\end{itemize}
		
		In both cases one can directly verify $y_i=f_i(x)$ for all $0\le i\le n$. 
		
		To see that the image is linearly isomorphic to boundary of Lorentz cone, let $z_0=\frac{y_0}{2},z_1=y_1-\frac{y_0}{2},z_i=y_i$ for all $2\le i\le n$ and with this substitution one can verify that $\{(y_0,\ldots,y_n):y_0y_1=\sum_{i=1}^n y_i^2,y_0\ge 0,y_1\ge 0\}=\{(z_0,\ldots,z_n):z_0^2=\sum_{i=1}^n z_i^2,z_0\ge 0 
		\}$. 
	\end{proof}
	
	Proposition \ref{prop:bd_soc} shows that the image of this quadratic map is the boundary of a full-dimensional closed pointed convex cone. We now show that the image of such a set under any linear map with nontrivial kernel must be convex. 
	
	\begin{lemma}\label{lem:proj_bdcvx}
		Let $C\subseteq \R^n$ be a full-dimensional closed convex set which does not contain lines. Let $\pi:\R^n\to\R^m$ be any linear mapping with nontrivial kernel. Then $\pi(\partial C)=\pi(C)$, and therefore $\pi(\partial C)$ is convex. 
	\end{lemma}
	
	\begin{proof}
		Let $x$ be a point in the interior of $C$. It suffices to show $(x+\ker\pi)\cap \partial C\ne\emptyset$, as $\pi(x)=\pi(z)$ for all $z\in x+\ker\pi$. Let $\ell$ be any one-dimensional linear subspace of $\ker \pi$. Since $C$ does not contain lines, there exists $y\in x+\ell$ that is not contained in $C$. Thus there exists $z$ on the segment $[x,y]$ that is on boundary of $C$. 
	\end{proof}
	
	\begin{remark}
		This statement may not be true when $C$ contains a line. For example let $C=[-1,1]\times \R\subseteq \R^2$ and $\pi$ be projection onto the first coordinate. Then $\pi(\partial C)=\{-1,1\}$, while $\pi(C)=[-1,1]$. 
	\end{remark}
	
	We now show that the following special class of quadratic maps have hidden hyperplane convexity: one of the maps is given by the identity matrix, and the rest are all products of a fixed variable with linear forms, and the linear forms do not span $\mathbb{R}^n$. \blue{\cite{argue2023necessary} studies a related problem on symmetric matrices of linear forms. }
	
	\begin{theorem}\label{thm:cvx}
		Fix integers $n\ge 2,m\ge 1$ and consider the following $m+1$ quadratic forms on $\R^n$: $f_0(x_1,\ldots,x_n)=\sum_{i=1}^n x_i^2,f_j(x_1,\ldots,x_n)=x_1 \ell_j$ for all $1\le j\le m$, where each $\ell_j=(v^{(j)})^{\top}x=\sum_{i=1}^n (v^{(j)})_i x_i$ is a linear form in $x_1,\ldots,x_n$. Let $\L=\span\{v^{(1)},\ldots,v^{(m)}\}$. If $\L\ne\R^n$, then the set $\{(f_0(x),\ldots,f_m(x)):x\in\R^n\}\subseteq \R^{m+1}$ is convex. 
	\end{theorem}
	
	\begin{proof}
		Let $\pi:\R^{n+1}\to\R^{m+1}$ be the linear map $\pi(y_0,y_1,\ldots,y_n)=(y_0,z_1,\ldots,z_m)$ where $z_j=\sum_{i=1}^n (v^{(j)})_i y_i$ for all $1\le j\le m$. Then $\ker\pi=\{(0,w):w\in\L^{\perp}\}$ is nontrivial, and $\{(f_0(x),\ldots,f_m(x)):x\in\R^n\}=\pi\{(\sum_{i=1}^n x_i^2,x_1^2,x_1x_2,\ldots,x_1x_n):x\in\R^n\}$. Thus convexity follows from Proposition \ref{prop:bd_soc} and Lemma \ref{lem:proj_bdcvx}. 
	\end{proof}
	
	We have the following immediate Corollary when $m<n$, since in this case the linear forms do not span $\mathbb{R}^n$.
	\begin{corollary}\label{cor:cvx}
		Fix integers $n>m\ge \blue{1}$ and consider the following $m+1$ quadratic forms on $\R^n$: $f_0(x_1,\ldots,x_n)=\sum_{i=1}^n x_i^2,f_j(x_1,\ldots,x_n)=x_1 \ell_j$ for all $1\le j\le m$, where each $\ell_j$ is a linear form in $x_1,\ldots,x_n$. Then the set $\{(f_0(x),\ldots,f_m(x)):x\in\R^n\}\subseteq \R^{m+1}$ is convex. 
	\end{corollary}
	
	Corollary \ref{cor:cvx} can be extended to quadratic forms
	defined on an arbitrary finite dimensional real vector space, 
	after choosing suitable bases. More generally, we only need for the first quadratic form to be positive definite, and the remaining ones to have a common linear factor. 
	
	\begin{proposition}\label{prop:cvx_quadforms}
		Fix integers $n>m\ge \blue{1}$ and $V$ be any 
		real vector space of dimension $n$. 
		Let $f_0,\ldots, f_m:V\to \R$ be quadratic forms on $V$ such that $f_0$ is positive definite (on $V$), and there exists linear form $\ell:V\to\R$ such that for all $1\le i\le m$, $f_i(x)=\ell(x)\ell_i(x)$ for some linear form $\ell_i:V\to\R$. Then the set $\{(f_0(x),\ldots,f_m(x)):x\in V\}\subseteq \R^{m+1}$ is convex. 
	\end{proposition}
	
	\begin{proof}
		If $\ell$ is identically zero on $V$, then we have $\{(f_0(x),\ldots,f_m(x)):x\in V\}=\{(c,0,\ldots,0): c\ge 0 \}$ which is clearly convex. Thus from now on assume $\ell$ is a nonzero linear form. 
		
		Our goal is to choose basis for $V$ in which $f_0$ becomes identity matrix and $\ell(x)$ becomes $x_1$, so that we can apply Corollary \ref{cor:cvx}. Let $B:V\times V\to\R$ be the symmetric bilinear form associated with $f_0$, i.e., $B(x,y)=\frac{1}{4}(f_0(x+y)-f_0(x-y))$. Since $f_0$ is positive definite, $B$ defines an inner product on $V$, by $\langle x,y\rangle_B= B(x,y)$. 
		Consider $\ker\ell$, which is a dimension $n-1$ linear subspace of $V$. Let $\{v_2,\ldots,v_m\}$ be a basis of $\ker \ell$ that is orthonormal with respect to inner product $\langle \cdot,\cdot\rangle_B$, which can be found using Gram-Schmidt in this inner product. Append $v_1$ so that $\{v_1,\ldots,v_m\}$ is an orthonormal basis of $V$ with respect to inner product $\langle \cdot,\cdot\rangle_B$. 
		
		Then by orthonormality we have 
		$$B(v_i,v_j)= \left\{\begin{array}{cc} 1& i = j\\ 0 & i \neq j \end{array}\right.,$$
		and $\ell(v_i)\ne 0$ if and only if $i=1$, since $v_1\notin \ker \ell$ and all other $v_i$ are in the kernel. Let $g_0,\ldots,g_m$ be quadratic forms on $\R^n$ such that $g_i(x_1,\ldots,x_n)=f_i(x_1v_1+\ldots+x_nv_n)$ for all $0\le i\le m$. Then $\{(g_0(x),\ldots,g_m(x)):x\in \R^n\}=\{(f_0(x),\ldots,f_m(x)):x\in V\}$ and one can verify $g_0(x)=\sum_{i=1}^n x_i^2$, $g_i(x)=x_1\ell_i'(x)$ for all $1\le i\le m$, where $\ell_i'(x)=\ell(v)\ell_i(x_1v_1+\ldots+x_nv_n)$ is a linear form in $x$. Thus by Corollary \ref{cor:cvx}, $\{(g_0(x),\ldots,g_m(x)):x\in \R^n\}$ is convex, and same holds true for $\{(f_0(x),\ldots,f_m(x)):x\in V\}$. 
	\end{proof}
	
	Proposition \ref{prop:cvx_quadforms} leads to a quick proof of Theorem~\ref{thm:HHTnontriv} that such quadratic maps also have hidden hyperplane convexity.

	\begin{proof}[Proof of Theorem~\ref{thm:HHTnontriv}]
		Apply Proposition \ref{prop:cvx_quadforms} to all quadratic forms restricted to $H$, and observe that $f_0|_H$ is positive definite, and $f_i(x)=\ell(x)\ell_i(x)$ for all $x\in H$ and $1\le i\le m$, where $\ell,\ell_i$ are viewed as linear forms on $H$. 
	\end{proof}

	
	\section{Proof of Theorem~\ref{thm:more_quadratics}}\label{sec:more_quadratics}
	There are two key results used in proof of Theorem~\ref{thm:more_quadratics} for the three quadratics PDLC case in \cite{dey2021obtaining}, which we will also use. The first is from~\cite{yildiran2009convex}, which characterizes when a homogeneous quadratic function has exactly one negative eigenvalue. 
	
	\begin{theorem}\label{thm:SCC}
		Let $P$ be any $(n+1)\times (n+1)$ symmetric matrix and $\mathcal{P} =\{ x\in \mathbb{R}^{n+1}\, : \, x^{\top} P x < 0\} \neq \emptyset$. The following are equivalent:
		\begin{enumerate}
			\item There exists a linear hyperplane that does not intersect $\mathcal{P}$.
			\item $P$ has one negative eigenvalue.
			\item $\mathcal{P}$ is an open semi-convex cone (SCC), i.e., a union of two disjoint open convex cones which are symmetric reflections of each other with respect to the origin. 
		\end{enumerate}
	\end{theorem}
	
	This result implies the following characterization of elements in $\Omega$. 
	
	\begin{lemma}\label{lem:Omega}Suppose that $S\ne\emptyset$, and let 
		$$\Omega = \{\lambda \in \mathbb{R}^m_+\setminus\{0\} : \conv(S)\subseteq S_{\lambda} \textup{ and }Q_{\lambda} \textup{ has at most one negative eigenvalue}\}.$$ Let $\lambda\in\mathbb{R}^m_+\setminus\{0\}$ be such that $Q_{\lambda}$ has at most one negative eigenvalue.  Then $Q_{\lambda}$ has exactly one negative eigenvalue, and $S_{\lambda}$ is either a convex set or a union of two disjoint convex sets. Furthermore, $\lambda\notin \Omega$ if and only if $S_{\lambda}$ is a union of two disjoint convex sets, and $S$ has nonempty intersection with both components. 
	\end{lemma}
	
	\begin{proof}
		If $Q_{\lambda}$ is PSD then $S\subseteq S_{\lambda}=\emptyset$. When $Q_{\lambda}$ has exactly one negative eigenvalue, using Theorem \ref{thm:SCC} the set $(S_{\lambda})^h=\{(x,x_{n+1}):(x,x_{n+1})^{\top} Q_{\lambda}(x,x_{n+1})<0 \}$ is a union of two disjoint open convex cones. Hence $S_{\lambda}\blue{\times \{1\}}=(S_{\lambda})^h \cap\{(x,x_{n+1}):x_{n+1}=1\}$ is either convex or a union of two \blue{disjoint} convex sets. Since $S_{\lambda}$ always contains $S$, the only way it fails to contain $\conv(S)$ is when $S_{\lambda}$ \blue{is a union of two disjoint convex sets} and $S$ has nonempty intersection with both. 
	\end{proof}
	
	We describe in more detail the set defined by a single quadratic inequality, whose matrix has exactly one negative eigenvalue. 
	
	\begin{lemma}\label{lem:single_quadratic}
		Let $f(x)=x^{\top}Ax+2b^{\top}x+c$ such that $Q=\begin{bmatrix}
			A & b\\b^{\top} & c
		\end{bmatrix}$ has exactly one negative eigenvalue. Let $S=\{x:f(x)<0\}$. Then $S$ is convex if $A$ is PSD, and $S\subsetneq \R^n$ unless $A=0,b=0$, and $c<0$. If $A$ is not PSD, then $S$ is union of two convex sets and $\conv(S)=\R^n$. 
	\end{lemma}
	
	\begin{proof}
		If $A$ is PSD the result is straightforward. If $A$ is not PSD we apply Theorem 1 in \cite{yildiran2009convex}, the characterization of convex hull defined by 2 quadratics, where $f_1=f$ and $f_2=-\|x\|_2^2-1$ with $Q_2=-I_{n+1}$. Then $\lambda A_1+(1-\lambda)A_2=\lambda A-(1-\lambda)I_n$ is never PSD for any $0\le\lambda\le 1$, which means $\conv(S_1)=\conv(S_1\cap S_2)=\R^n$.   
	\end{proof}
	
	The next result is a homogeneous separation lemma which was proved in \cite{dey2021obtaining}, which holds for arbitrary quadratics. Recall that the homogenization $S^h$ of $S$ is defined as follows:
	$$
	S^h=\{(x,x_{n+1}): \hf_i(x,x_{n+1})=x^{\top}A_ix+2(b_i^{\top}x)x_{n+1}+c_ix_{n+1}^2<0,i\in [m] \}. 
	$$
	
	\begin{lemma}[Lemma 5.4 in \cite{dey2021obtaining}]\label{lem:homog_separation}
		Let $\alpha^{\top}x<\beta$ be a valid inequality for $\conv(S)$. If $\conv(S)\ne\R^n$, then $\{(x,x_{n+1}): \alpha^{\top}x=\beta x_{n+1} \}\cap S^h=\emptyset$. 
	\end{lemma}
	
	We are now ready to prove our main theorem of this section.
	
	\begin{proof}[Proof of Theorem \ref{thm:more_quadratics}]
		By definition of $\Omega$ we automatically have $\conv(S)\subseteq \bigcap_{\lambda\in \Omega} S_{\lambda}$. For the other direction \blue{$\bigcap_{\lambda\in\Omega}S_{\lambda}\subseteq \conv(S)$, it suffices to show that for any fixed} $\tilde{x}\notin \conv(S)$, \blue{$\tilde{x}\notin \bigcap_{\lambda\in\Omega}S_{\lambda}$, or in other words, }there exists $\lambda\in \Omega$ where $\blue{\tilde{x}}\notin S_{\lambda}$. 
		
		By separation theorem for convex sets, there exist $\alpha\in\R^n,\beta\in\R$ so that $\alpha^{\top}x<\beta$ for all $x\in\conv(S)$ and $\alpha^{\top}\blue{\tilde{x}}=\beta$. Lemma \ref{lem:homog_separation} states that $H\cap S^h=\emptyset$ where $H=\{(x,x_{n+1}): \alpha^{\top}x=\beta x_{n+1} \}$ is a hyperplane in the homogenized space. 
		
		Let $\hf:\R^{n+1}\to\R^m$ be the associated quadratic map, i.e., $\hf=(\hf_1,\ldots,\hf_m)$ where $\hf_i(x,x_{n+1})=x^{\top}A_ix+2(b_i^{\top}x)x_{n+1}+c_ix_{n+1}^2$. By definition of hidden hyperplane convexity, $\im \hf|_H$ is convex, and $H\cap S^h=\emptyset$ means that $\im \hf|_H$ does not intersect the open negative orthant $\{y\in\R^m: y_i<0,i\in [m]\}$. Since both sets are convex, there exists a separating hyperplane, that is, there exists $\lambda\in\R^m\setminus \{0\},\mu\in\R$ such that
		\begin{equation*}\begin{aligned}
				\im \hf|_H & \subseteq \{y\in\R^m :  \lambda^{\top}y\ge \mu \}\\
				\{y\in\R^m:  y_i<0, i\in [m] \} & \subseteq \{y\in\R^m : \lambda^{\top}y \le \mu \}. 
		\end{aligned}\end{equation*}
		
		Then note that
		\begin{itemize}
			\item $\lambda\in\R^m_+$. Otherwise suppose there exists some $j$ where $\lambda_j<0$ for contradiction. Then for any $M>0$ let $v_M=-Me_j-\sum_{i = 1}^m e_i$, where $e_i$ is the $i$-th standard basis vector. Then $v_M\in\{y\in\R^m: y_i<0, i\in [m] \}$ for any $M>0$, and $\lambda^{\top}v_M=-M\lambda_j-\sum_{i =1}^m \lambda_i v_i$. Since $\lambda_j<0$ and $\sum_{i =1}^m \lambda_i v_i,\mu$ is constant, for sufficiently large $M$ we have $\lambda^{\top}v_M>\mu$, which is contradiction. 
			
			\item $\mu=0$. Since $0\in \im \hf|_H$ we have $\mu\le 0$. For the other inequality, for any $M>0$ let $w_M=-\frac{1}{M}\sum_{i = 1}^m e_i\in \{y\in\R^m: y_i<0, i\in [m] \}$. We have $\mu\ge \lambda^{\top} w_M =-\frac{1}{M}\sum_{i = 1}^m \lambda_i $. Let $M\to\infty$ we get $\mu\ge 0$. 
		\end{itemize}
		
		Thus $\im \hf|_H\subseteq \{y\in\R^m\, : \, \lambda^{\top}y\ge 0 \}$ where $\lambda\in\R^m_+\setminus\{0\}$. This means for any $(x,x_{n+1})\in H$, $$\sum_{i=1}^m \lambda_i\hf_i(x_{n+1})=\sum_{i =1}^m \lambda_i \begin{bmatrix}
			x^{\top} & x_{n+1}
		\end{bmatrix}\begin{bmatrix}
			A_i & b_i\\b_i^{\top} & c_i
		\end{bmatrix}\begin{bmatrix}
			\blue{x} \\ x_{n+1}
		\end{bmatrix}\ge 0.$$ Theorem \ref{thm:SCC} then implies that the matrix $Q_{\lambda}=\sum_{i =1}^m \lambda_i \begin{bmatrix}
			A_i & b_i\\b_i^{\top} & c_i
		\end{bmatrix}$ has exactly one negative eigenvalue, and $(S_{\lambda})^h=\{(x,x_{n+1}):(x,x_{n+1})^{\top} Q_{\lambda}(x,x_{n+1})<0 \}$ where $Q_{\lambda}$ consists of two disjoint convex cones separated by $H=\{(x,x_{n+1}) : \alpha^{\top}x=\beta x_{n+1}\}$. Thus $(S_{\lambda})^h\cap \{(x,x_{n+1}) :\alpha^{\top}x<\beta x_{n+1}\}$ is convex, which geometrically is simply half of $(S_{\lambda})^h$. Also note that $S\times \{1\}=\{(x,1) \, :\, x\in S \}$ is contained in both $(S_{\lambda})^h$ and $\{\hat{x}\, : \, \alpha^{\top}x<\beta x_{n+1}\}$. Thus $\conv(S)\times \{1\}\subseteq (S_{\lambda})^h\cap \{\hat{x}\, : \, \alpha^{\top}x<\beta x_{n+1}\}$ as the right side is convex. Since $S_{\lambda}\blue{\times \{1\}}=(S_{\lambda})^h\cap \{\hat{x}\, : \,  x_{n+1}=1\}$, we have $\conv(S)\subseteq S_{\lambda}$, which concludes the proof. 
	\end{proof}

	
	\section{Certificates using aggregations when $S=\emptyset$ or $\conv(S)=\R^n$}\label{sec:emptyrn}
	We first examine the case where $S=\emptyset$ and prove Proposition \ref{prop:empty}.
	\begin{proof}[Proof of Proposition \ref{prop:empty}]
		$(\Leftarrow)$ If $Q_{\lambda}\succeq 0$ then $S_{\lambda}=\emptyset$, and the result follows since $S\subseteq S_{\lambda}$. 
		
		$(\Rightarrow)$ Assume $S=\emptyset$. We first show $S^h=\emptyset$. Since $S=\emptyset$, for any $x\in\R^n$, $f_i(x)\ge 0$ for some $i$. Thus if $t\ne 0$, then $\hf_i(x,t)=t^2 f_i(x/t)$ and thus $\hf_i$ cannot be simultaneously negative for all $i$. This means $S^h\subseteq\{(x,x_{n+1})\in\R^{n+1}: x_{n+1}=0 \}$, but $S^h$ is open and full dimensional if nonempty. Therefore we must have $S^h=\emptyset$. 
		
		This means intersection of $\im\hf$ with negative orthant is empty. Thus there exists $\lambda\in\R^m\setminus \{0\},\mu\in\R$ such that
		\begin{equation*}\begin{aligned}
				\im \hf &\subseteq \{y\in\R^m :  \lambda^{\top}y\ge \mu \}\\
				\{y\in\R^m:  y_i<0, i\in [m] \}&\subseteq \{y\in\R^m : \lambda^{\top}y \le \mu \}. 
		\end{aligned}\end{equation*}
		
		Rest of the proof is exactly the same (other than here we have $\im \hf$ instead of $\im\hf|_H$) as the step in proof of Theorem \ref{thm:more_quadratics} where we show $\lambda\in\R^m_+$ and $\mu=0$, which means $\sum_{i = 1}^m \lambda_i \hf_i(x,x_{n+1})\ge 0$ for all $(x,x_{n+1})\in\R^{n+1}$, and $Q_{\lambda}\succeq 0$. Namely: 
		
		\begin{itemize}
			\item $\lambda\in\R^m_+$. Otherwise suppose there exists some $j$ where $\lambda_j<0$ for contradiction. Then for any $M>0$ let $v_M=-Me_j-\sum_{i = 1}^m e_i$, where $e_i$ is the $i^{th}$ standard basis vector. Then $v_M\in\{y\in\R^m: y_i<0, i\in [m] \}$ for any $M>0$, and $\lambda^{\top}v_M=-M\lambda_j-\sum_{i =1}^m \lambda_i v_i$. Since $\lambda_j<0$ and $\sum_{i =1}^m \lambda_i v_i,\mu$ is constant, for sufficiently large $M$ we have $\lambda^{\top}v_M>\mu$, which is contradiction. 
			
			\item $\mu=0$. Since $0\in \im \blue{\hf}$ we have $\mu\le 0$. For the other inequality, for any $M>0$ let $w_M=-\frac{1}{M}\sum_{i = 1}^m e_i\in \{y\in\R^m: y_i<0, i\in [m] \}$. We have $\mu\ge \lambda^{\top} w_M =-\frac{1}{M}\sum_{i = 1}^m \lambda_i $. Let $M\to\infty$ we get $\mu\ge 0$. 
		\end{itemize}
	\end{proof}
	
	We now examine the case $\textup{conv}(S) = \mathbb{R}^n$ and prove Proposition \ref{prop:Rn}.
	
	\begin{proof}[Proof of Proposition \ref{prop:Rn}]
		Let $g_i(x)=x^{\top}A_i x,i\in [m]$ be the quadratic parts of $f_i$. Then $g_i$ are homogeneous, and we have $\hat{g}_i=g_i$ and $\im \hat{g}=\im \hf|_{\mathcal{E}}$. Using Proposition \ref{prop:empty} for $g_i$ we have $\{x\in\R^n: g_i(x)<0,i\in [m] \}=\emptyset$ if and only if $\sum_{i =1}^m \lambda_i A_i\succeq 0$ for some nonzero $\lambda\ge 0$. 
		
		Thus if such $\lambda$ does not exist, then there exists some $v\in \R^n$ where $v^{\top}A_iv=g_i(v)<0$ for all $1\le i\le m$. We now show that for any fixed $x\in\R^n$, $x+Mv,x-Mv\in S$ for some $M>0$, which then implies $x\in\conv(S)$ and hence $\conv(S)=\R^n$. We have \begin{equation*}\begin{aligned}
				f_i(x+Mv)&=M^2(v^{\top}A_iv)+2M(x^{\top}A_iv+b_i^{\top}v)+f_i(x),\\
				f_i(x-Mv)&=M^2(v^{\top}A_iv)-2M(x^{\top}A_iv+b_i^{\top}v)+f_i(x).
		\end{aligned}\end{equation*}
		
		Since $v^{\top}A_iv<0$ for all $i$, the leading coefficient is negative, and the function values become negative for sufficiently large $M$. To be more precise, in order for $f_i(x+Mv)<0,f_i(x-Mv)<0$ for all $i$, it suffices to take
		$$
		M>\max_i\left\{ \frac{|x^{\top}A_iv+b_i^{\top}v|+\sqrt{(x^{\top}A_iv+b_i^{\top}v)^2-(v^{\top}A_i v)f_i(x)}}{-v^{\top}A_iv} \right\}
		$$
		
		where $i$ ranges over wherever the expression inside square root is non-negative. If such $i$ does not exist then $M$ can take any positive real value. 
	\end{proof}
	
	\section{Application to sets defined by spheres and halfspaces}\label{sec:cheese}
	
	Using results from Section \ref{sec:HHC} we now study sets defined by linear and sphere constraints: consider quadratic functions $f_i(x)=x^{\top}A_i x+2b_i^{\top}x+c_i, \,\, 1\le i\le m$ where each $A_i$ is either $I,-I$ or zero, and let $S=\{x\in\R^n: f_i(x)<0,1\le i\le m \}$. We call such a set $S$ defined by spheres and halfspaces, since each $f_i$ is either  an affine linear function, or it defines the interior or exterior of a sphere. Let $P,Z,N$ be the index sets of interior of the sphere constraints, exterior of the sphere constraints and affine linear constraints respectively: $P=\{i\in [m]:A_i=I \}$, $Z=\{i\in [m]:A_i=0 \}$, $N=\{i\in [m]:A_i=-I \}$. 
	
	
	

	In terms of matrices, after taking $\ell(x)=x_1$ Theorem~\ref{thm:HHTnontriv} can be restated as follows. Note we may use Lemma \ref{lem:HHC_preserved} to change basis for the matrices, corresponding to choosing different linear function $\ell(x)$. 
	
	\begin{corollary}\label{cor:hyp_cvx_matrices}
		Fix integers $n>m+1,m\ge 2$ and let $Q_0,\ldots,Q_m$ be $n\times n$ symmetric matrices such that $Q_0$ is positive definite, and $Q_i$ has nonzero entries only in first row and column for all $1\le i\le m$. Then $Q_0,\ldots,Q_m$ satisfy hidden hyperplane convexity. 
	\end{corollary}
	
	We can use Theorem \ref{thm:more_quadratics} to show that the convex hull of a set defined by constraints described above is always given by good aggregations.
	
	\begin{proposition}\label{prop:spheres}
		Let $f_i,i\in[m]$ be linear or sphere constraints such that one of the two following conditions holds: 
		
		\begin{itemize}
			\item $\dim \span(Q_1,\ldots,Q_m)\le n-1$. 
			\item $\dim \span(Q_1,\ldots,Q_m)= n$ and PDLC condition holds. 
		\end{itemize}
		
		Suppose $\emptyset\subsetneq\conv(S)\subsetneq\R^n$. Then $\conv(S)$ is defined by good aggregations, i.e., $\conv(S)=\bigcap_{\lambda\in\Omega}S_{\lambda}$, where $\Omega=\{\lambda\in\R_+^m\setminus\{0\}:\sum_{i=1}^m \lambda_i A_i\succeq 0 \}$ and $S_{\lambda}=\{x:\sum_{i=1} \lambda_i f_i(x)<0 \}$. 
	\end{proposition}
	
	Note that given the special structure of the constraints, the description of $\Omega$ can be greatly simplified compared to the general case. This is due to the fact that each $A_i$ is either $I,-I$ or 0, and therefore $Q_{\lambda}=\sum_{i=1}^m \lambda_i\begin{bmatrix}
		A_i & b_i\\b_i^{\top} & c_i
	\end{bmatrix}$ has at most one negative eigenvalue if and only if $\lambda_i A_i\succeq 0$. In this case $S_{\lambda}$ is automatically convex. In fact, the set $\Omega\cup\{0\}$ has a polyhedral description: $\Omega\cup\{0\}=\{\lambda\ge 0:\sum_{i\in P} \lambda_i \ge \sum_{i\in N} \lambda_i\}$. 
	
	\begin{proof}[Proof of Proposition \ref{prop:spheres}]
		When PDLC condition does not hold and there are $m\le n-1$ constraints, we can always add a trivial constraint $f_0=-1-\sum_{i=1}^n x_i^2$ with $Q_0=-I$, which does not change $S$. Thus from now on we assume $\dim \span(Q_1,\ldots,Q_m)\le n$ and PDLC holds. We show that there exist symmetric matrices $Q'_0,\ldots,Q'_{m-1}$ such that $\span(Q_1,\ldots,Q_m)= \span(Q'_0,\ldots,Q'_{m-1})$, $Q'_0\succ 0$ and $A'_i=0$ for all $1\le i\le m-1$. Then $Q_1,\ldots,Q_m$ satisfy HHC by Corollary \ref{cor:hyp_cvx_matrices} and Lemma \ref{lem:HHC_preserved}. 
		
		Such $Q'_0,\ldots,Q'_{m-1}$ can be chosen as follows: $Q'_0$ is the linear combination of $Q_1,\ldots,Q_m$ that is positive definite. Since $A_i=I,-I$ or 0, upon rescaling we may assume $A'_0=I$. Upon relabeling assume coefficient of $Q_m$ in the linear combination of $Q_1, \dots, Q_m$ that produces $Q'_0$ is nonzero. Now the $Q'_i$s (other than $Q'_0$) can be chosen to be
		
		$$Q'_i=\begin{cases}
			Q_i & \textup{ if }A_i=0\\
			Q_i-Q'_0 & \textup{ if }A_i=I\\
			Q_i+Q'_0 & \textup{ if }A_i=-I
		\end{cases},1\le i\le m-1$$
		
		Then clearly $A'_i=0$ for all $1\le i\le m$, and each $Q'_i$ is a linear combination of $Q_1,\ldots,Q_m$. Conversely, all $Q_1,\ldots,Q_{m-1}$ are linear combinations of $Q'_0,\ldots,Q'_{m-1}$, and $Q_m$ is a linear combination of $Q'_0,Q_1,\ldots,Q_{m-1}$ and hence $Q'_0,\ldots,Q'_{m-1}$. 
	\end{proof}

	
	

	Recall from Proposition \ref{prop:empty} that $S=\emptyset$ if and only if there exists nonzero $\lambda\ge 0$ such that $Q_{\lambda}\succeq 0$, which can be checked using an semidefinite program. In general, we do not have a necessary and sufficient condition for $\conv(S)=\R^n$. In the special case of sphere and linear constraints, it is easy to determine whether $\conv(S)=\R^n$ by checking the types of sphere constraints. Recall $P,Z,N$ denote index sets where $A_i=I,0,-I$ respectively. 
	
	\begin{lemma}\label{lem:Rn_sphere}
		$\conv(S)\ne\R^n$ if and only if either $P\ne\emptyset$ or there exists $i\in Z$ such that $b_i\ne 0$ or $b_i=0$ and $c_i\ge 0$. 
	\end{lemma}
	
	\begin{proof}
		$(\Rightarrow)$ We show the contrapositive. Suppose $P=\emptyset$ and for all $i\in Z$ we have $b_i=0$ and $c_i<0$. Then complement of each $S_i$ is bounded, and thus complement of $S$ is also bounded, which means $\conv(S)=\R^n$. 
		
		$(\Leftarrow)$ If such $i$ exists then $S_i$ is convex and not $\R^n$. Thus $\conv(S)\ne\R^n$ as $S\subseteq S_i$. 
	\end{proof}
	
	\subsection{Finiteness for number of good aggregations}
	
	Unlike in the general case where we need further conditions to ensure that the convex hull is given by finitely many good aggregations, for linear and sphere constraints if the convex hull is given by good aggregations, then finitely many will always suffice. The number of aggregations needed will depend on size of $P,Z,N$, which count the number of constraints where $A_i=I,0,-I$ respectively. 
	
	\begin{proposition}\label{prop:finiteness_sphere}
		Let $S$ be defined using linear and sphere constraints such that $\conv(S)$ is given by good aggregations. Then $\conv(S)$ can be described by at most $|P||N|+|P|+|Z|$ aggregations. Furthermore each good aggregation is either of the form $f_i$ for some $i\in P\cup Z$ or $f_i+f_j$ for some $i\in P,j\in N$.  
	\end{proposition}
	
	\begin{proof}
		First observe that $f_i,i\in P\cup Z$ and $f_i+f_j,i\in P,j\in N$ are good aggregations, as their leading $n\times n$ principal submatrices are PSD. Let $f_{\lambda}=\sum_{i=1}^m \lambda_i f_i$ be any good aggregation. We claim that there exist nonnegative coefficients $\{\gamma_{ij},i\in P,j\in N\},\{\alpha_i,i\in P\},\{\beta_k, k\in Z\}$ such that \begin{equation}\label{eqn:pzn}
			f_{\lambda}=\sum_{i\in P,j\in N} \gamma_{ij} (f_i+f_j)+\sum_{i\in P} \alpha_i f_i+\sum_{k\in Z}\beta_k f_k.
		\end{equation} The above equation would show that $f_{\lambda}$ is an aggregation of $f_i,i\in P\cup Z$ and $f_i+f_j,i\in P,j\in N$. This implies that the constraint $f_{\lambda} < 0$ is dominated by the intersection of the constraints $f_i + f_j < 0,$ $(i,j) \in P \times N$, $f_i, i \in P $, and $f_k < 0$, $k \in Z$, giving us the upper bound of $|P||N|+|P|+|Z|$.
		
		Now we prove the claim about \eqref{eqn:pzn}. Since $f_{\lambda}$ is a good aggregation, we must have $\lambda\ge 0$ and $\sum_{i\in P}\lambda_i\ge \sum_{j\in N}\lambda_j$. If $\sum_{j\in N}\lambda_j=0$ then $\lambda_j=0$ for all $j\in N$, and we may choose $\gamma_{ij}=0$ for all $i\in P,j\in N$,$\alpha_i=\lambda_i$ for all $i\in P$ and $\beta_k=\lambda_k$ for all $k\in Z$. 
		
		From now on assume $\sum_{j\in N}\lambda_j>0$. For all $i\in P$ we let $$\mu_i=\lambda_i\frac{\sum_{j\in N}\lambda_j}{\sum_{i\in P}\lambda_i},$$ which satisfies $0 \le \mu_i\le \lambda_i$ for all $i\in P$ and $\sum_{i\in P}\mu_i=\sum_{j\in N}\lambda_j$. Then (\ref{eqn:pzn}) holds true if we let $$\gamma_{ij}=\frac{\mu_i\lambda_j}{\sum_{j\in N}\lambda_j},i\in P,j\in N, \quad \alpha_i=\lambda_i-\mu_i,i\in P, \quad \beta_k=\lambda_k,k\in Z. $$\end{proof}
	
	\section{Diagonal Inequalities}\label{sec:diag}
	For $x,y\in\R^n$, we let $x\circ y\in\R^n$ denote the element-wise product $(x_1y_1,...,x_ny_n)$. Given a vector $u \in\R^k$, we let $\textup{Diag}(u)$ be the $k \times k$ diagonal matrix with $(i,i)$ diagonal entry equal to $u_i$. 
	
	Let $Q_1,\ldots, Q_m$ be diagonal $(n+1)\times (n+1)$ matrices, and consider the set $$S=\{x\in\R^n: \begin{bmatrix}
		x\\1
	\end{bmatrix}^{\top} Q_i\begin{bmatrix}
		x\\1
	\end{bmatrix}< 0, 1 \leq i\le m \}.$$ 
	
	Our goal is to describe $\conv(S)$, and it is natural to consider the following ``open-polyhedron": let $A\in\R^{m \times n},b\in\R^m$ be defined such that $Q_i=\textup{Diag}(a_{i1},...,a_{in},-b_i)$, then clearly $S=\{x\in\R^n: x\circ x\in \P \}$ where $\P=\{x\in\R^n: Ax < b \}$ is a open-polyhedron, and in particular it is open in the topological sense, since it is given by intersection of finitely many open halfspaces. The usual notion of polyhedron given by closed linear inequalities is sometimes referred to as closed polyhedron. 
	
	Now we define a new set $\P'=\{y \in\R^n: \exists x,Ax<b,y  < x \}$. We first show that $\P'$ describe the element-wise squares of the convex hull. 
	
	\begin{proposition}\label{prop:conv}
		$\conv(S)=\{x\in\R^n: x\circ x\in\P'\}$ and $\P'\supseteq \P$. 
	\end{proposition}
	\begin{proof}
		\blue{Let $\bar{1}\in\R^n$ be the all-one vector.} First observe that $\P\subseteq \P'$: If $\hat{x} \in \P$, then there exists $\epsilon >0$ such that $\hat{x} + \epsilon \cdot \bar{1} \in \P$, since $\P$ is open.  Since $\hat{x} < \hat{x} + \epsilon \cdot \bar{1}$, this implies $\hat{x} \in \P'$. 
		
		Since $\P\subseteq \P'$ we have $S\subseteq \{y\in\R^n: y\circ y\in\P'\}$. Also observe that if $x\in S$, then $u\in S$ where $|u_i|=|x_i|$ for all $i\in [n]$. Thus, given $x \in S$, if $|y_i| < |x_i|$ for all $i \in [n]$ (or equivalently $y \circ y < x\circ x$), then $y \in \textup{conv}(S)$. This shows that $\{y\in\R^n: y\circ y\in\P'\}\subseteq \conv(S)$, i.e. ,
		$$S \subseteq \{y\in\R^n: y\circ y\in\P'\}\subseteq \conv(S).$$
		
		It remains to show that $\{y\in\R^n: y\circ y\in\P'\}$ is convex. Let $u, v$ be such that $u\circ u, v\circ v\in \P'$. This means there exists $u',v'$ such that $u \circ u < u'\circ u', v\circ v < v'\circ v'$ and $u'\circ u', v'\circ v'\in \ P$. Note that $u \circ u < u' \circ u', v\circ v < v'\circ v'$ is equivalent to $|u_i| < |u'_i|, |v_i|< |v'_i|$ for all $i \in [n]$. 
		
		Fix any $0\le \lambda\le 1$, and let $w=\lambda u+(1-\lambda) v$. We now show $w\circ w\in \P'$ which then completes the proof. Since $f(t)=t^2$ is a convex function, we have
		\begin{equation*}
			\begin{aligned}
				(w\circ w)_i & =(\lambda u_i+(1-\lambda)v_i)^2\\
				& \le \lambda u_i^2+(1-\lambda)v_i^2\\
				& < \lambda (u'_i)^2+(1-\lambda)(v'_i)^2,
			\end{aligned}
		\end{equation*}
		
		which shows $w \circ w < \lambda( u'\circ u')+(1-\lambda)(v'\circ v')$. Since $u'\circ u', v'\circ v' \in \P$ and $\P$ is convex, we have $\lambda( u'\circ u')+(1-\lambda)(v'\circ v')\in\P$ which shows $w\circ w\in \P'$.
	\end{proof}
	
	The following observation about $\P$ and $\P'$ \blue{is} useful in further analysis.
	
	\begin{observation}\label{obs:diagonal_closure}
		$\P$ and $\P'$ are always open. If $\P\ne\emptyset$ then its closure is given by
		$\overline{\P}:= \{ x \in \mathbb{R}^n : Ax \le b\}$, in which case $\overline{\P'}=\{y \in\R^n: \exists x,Ax\le b,y \le x \}$ is a closed polyhedron given by $\overline{\P'}=\{x\in\R^n:Gx\le h\}$ for some $G\in\R^{k\times n},h\in\R^k$ where $k$ is some positive integer. Then $\P'=\text{int}(\overline{\P'})=\{x\in\R^n:Gx<h\}$ is an open-polyhedron. 
	\end{observation}
	
	\begin{proof}$\P$ is clearly open since it is intersection of finitely many open halfspaces. $\P'$ is open since it is the projection of an open set. 
		
		From now on assume $\P\ne\emptyset$. Clearly $\{ x \in \mathbb{R}^n : Ax \le b\}$ is closed and contains $\P$ and hence $\overline{\P}$. For the reverse inclusion, Let $\hat{x}$ satisfies $A\hat{x}\le b$ and fix arbitrary $u \in \P$. Then note that for $\lambda \in (0,1)$ we have $\lambda \hat{x} + ( 1- \lambda) u \in \P$. By selecting $\lambda$ arbitrarily close to $1$, we can obtain a point arbitrarily close to $\hat{x}$, which means it is a limit point of $\P$ and hence belongs to its closure. 
		
		We now study closure of $\P'$. Clearly $\{y \in\R^n: \exists x,Ax\le b,y \le x \}$ contains $\P'
		$, and is a closed polyhedron since it is the projection of a closed polyhedron $\{(x,y)\in\R^{2n}:Ax\le b,y\le x\}$ onto the $y$ coordinates. To show $\{y\in\R^n:\exists x,Ax\le b,y\le x\}\subseteq \overline{\P'}$, take any $\hx,\hy\in\R^n$ with $A\hx\le b,\hy\le \hx$. Since $\P\ne\emptyset$ there exists $u\in\P$. Since $\P$ is open there exists $\epsilon >0$ such that $v=u + \epsilon \cdot \bar{1} \in \P$, which satisfies $Av<b,u<v$. Therefore for all $0<\lambda<1$ we have $$A(\lambda\hx+(1-\lambda)v)=\lambda A\hx+(1-\lambda)Av<b,\lambda\hy+(1-\lambda)u<\lambda\hx+(1-\lambda)v. $$Therefore $\lambda\hy+(1-\lambda)u\in\P'$. Letting $\lambda\to 1$ we get $\hy\in \overline{\P'}$. 
		
		Now $\P'\supseteq \P\ne\emptyset$ is open, and its closure $\overline{\P'}$, say $\{x:Gx\le h\}$, is a closed polyhedron with nonempty interior. First half of observation states that closure of $\{x:Gx<h\}$ is $\{x:Gx\le h\}$ if $\{x:Gx<h\}\ne\emptyset$. Furthermore $\{x:Gx<h\}$ cannot be empty otherwise $\{x:Gx\le h\}$ is not full-dimensional. Since both $\P'$ and $\{x:Gx<h\}$ are (topologically) open with same closure, they must coincide. 
	\end{proof}

	
	Now we state the necessary and sufficient condition for $\P'\supseteq\R^n_+$. Geometrically, the condition states that the recession cone of $\P$, same as the recession cone of $\overline{\P}$, contains a strictly positive vector. 
	
	\begin{proposition}\label{prop:diagonal_Rn}
		Suppose $\P \cap \mathbb{R}^n_{+} \ne\emptyset$. Then $\P' \supseteq \R^n_+$ if and only if $Ax \leq 0,x>0$ is feasible, or equivalently there does not \blue{exist} $u\ge 0$ such that $u^{\top}A\ge 0$ and is nonzero. 
	\end{proposition}

	\begin{proof}
		
		Observe $Ax\le 0,x>0$ is infeasible if and only if $Ax\le 0,-x\le -\bar{1}$ is infeasible. By Farkas lemma, this happens if and only if there exist $u,v\ge 0$ such that $u^{\top}A-v^{\top}=0,-v^{\top}\bar{1}<0$, or equivalently there exists $u\ge 0$ such that $u^{\top}A\ge 0$ and is nonzero. 
		
		$(\Rightarrow)$ We show the contrapositive and assume there exists $u\ge 0$ such that $u^{\top}A\ge 0$ and is nonzero. Fix index $i$ where $i^{th}$ coordinate of $u^{\top}A$ is strictly positive. Then for sufficiently large $\lambda$ we have $\lambda u^{\top}Ae_i>u^{\top}b$. We now claim $\lambda e_i\notin \P'$ which completes the proof of this direction. By definition of $\P'$ it suffices to show for any $w>\lambda e_i$ we have $w\notin \P$. Suppose for contradiction that $w\in \P$ for some $w>\lambda e_i$. Then we would have $Aw<b$ since $w\in\P$, and $u^{\top}Aw\le u^{\top}b$ since $u\ge 0$. But then we would also have $u^{\top}A (\lambda e_i)\le u^{\top}b$ since $w>\lambda e_i$ and $u^{\top}A\ge 0$, contradicting $\lambda u^{\top}A e_i> u^{\top}b$. 
		
		

		$(\Leftarrow)$ Choose $w\in \P\cap\R^n_+$ and $v>0$ such that $Av\le 0$. Then for all $\lambda>0$, $A(w+\lambda v)<0$, i.e., $w+\lambda v\in \P$. Observe for any $x\ge 0$ we have $x < w+\lambda v$ for sufficiently large $\lambda$, which shows $x\in\P'$. 
	\end{proof}

	Observation \ref{obs:diagonal_closure} states that $\P'$ is an open polyhedron when $\P\ne \emptyset$. The following proposition states that when $\P\cap\R_+^n\ne\emptyset$, then the defining inequalities of $\P'$ are precisely aggregations of defining inequalities of $\P$, where all coefficients in the aggregation are nonnegative. 
	
	\begin{proposition}\label{prop:facets} Let 
		$\P \cap \mathbb{R}^n_{+} \neq \emptyset$.  Then $\P' = \{ y : \blue{(\alpha^{(i)})^{\top}}y < \beta^i, i \in [k]\}$ for some integer $k$ and 
		\begin{itemize}
			\item $\alpha^{(i)} \geq 0$ and $b^{(i)} \geq 0$ for all $i \in [k]$. 
			\item There exists $\lambda^{(i)} \geq 0$, such that $(\alpha^{(i)})^{\top} = (\lambda^{(i)})^{\top}A$ and $\beta^{(i)} = (\lambda^{(i)})^{\top}b$.
		\end{itemize}
	\end{proposition}
	\begin{proof}
		Consider the lifted open-polyhedron $Q$ and its closure given by Observation \ref{obs:diagonal_closure}: 
		$$Q := \{(x,y)\in\R^{2n}: Ax < b, y < x\},\overline{Q} := \{(x,y)\in\R^{2n}: Ax \le b, y \le x\}.$$
		
		Observation \ref{obs:diagonal_closure} states that $\overline{\P'}=\{y\in\R^n: \exists x:Ax\le b,y\le x\}$ is the projection of $\overline{Q}$ onto the $y$ coordinates, which means $\overline{\P'}$ has description $\overline{P'}=\{ y : (\alpha^{(i))^{\top}}y \le \beta^i, i \in [k]\}$ for some integer $k$, where each inequality is facet defining. Observation \ref{obs:diagonal_closure} also shows that $\P'$ is given by the corresponding open inequalities $\P' = \{ y : (\alpha^{(i))^{\top}}y < \beta^i, i \in [k]\}$. 
		
		Now it remains to show these inequalities have the desired form. Consider a facet defining inequality $\alpha ^{\top}x \leq \beta$ for  $\overline{\P'}=\textup{proj}_y(\overline{Q})$. We can consider the following associated LP and its dual:
		\begin{eqnarray*}
			&\textup{max} &\alpha^{\top} y \\
			&\textup{s.t.} &Ax \leq b \\
			&& y - x \leq 0. 
		\end{eqnarray*}
		\begin{eqnarray*}
			&\textup{min} & \lambda^{\top}b \\
			&\textup{s.t.} & \lambda^{\top}A  - \mu^{\top}I  = 0\\
			&& \mu^{\top}I = \alpha^{\top}\\
			&& \lambda \geq 0, \mu \geq 0.
		\end{eqnarray*}
		Note that since the primal is feasible and bounded, this implies the dual is feasible and bounded. In particular, let  $(\lambda^*, \mu^*)$ be the optimal dual solution. Then, from the second constraint of the dual $\alpha = \mu^* \geq 0$. Since $\P \cap \mathbb{R}^n_{+} \neq \emptyset$, we have that $\beta \geq 0$. 
		
		Also note the first constraint of the dual now shows that $(\lambda^*)^{\top} A = \mu^* = \alpha$ and the objective of the dual optimal solution shows that $(\lambda^*)^{\top}b = \beta$. 
	\end{proof}
	
	Now we are ready to describe the convex hull when all quadratics are diagonal. 
	
	
		
		
	
	
	\begin{proof}[Proof of Theorem \ref{thm:diagonal}]
		Let $A\in\R^{m \times n},b\in\R^m$ be defined such that $Q_i=\textup{diag}(a_{i1},...,a_{in},-b_i)$. In this definition $S=\{x\in\R^n: x\circ x\in\P \}$ where $\P=\{x\in\R^n: Ax < b \}$. Let $\P'=\{y\in\R^n: \exists x\in \P,y < x\}$. As shown in Proposition \ref{prop:conv}, $\conv(S)=\{x\in\R^n: x\circ x\in\P'\}$. We have three cases
		
		\begin{itemize}
			\item $S = \emptyset$. Since $x\circ x\ge 0$ for all $x$, $S=\emptyset$ if and only if $Ax<b,x\ge 0$ is infeasible. By Motzkin's transposition theorem \cite{schrijver1998theory}, this happens if and only if there exists $u,v\ge 0,u^{\top}A-v^{\top}=0$ such that $u\ne 0$ and $u^{\top}b\le 0$, i.e. there exists nonzero $u\ge 0$ such that $u^{\top}A\ge 0,u^{\top}b\le 0$. Note if $u=0$ then $u^{\top}b\ge 0$ is automatically true. Now observe $u^{\top}A\ge 0,u^{\top}b\le 0$ is equivalent to $\sum_{i=1}^m u_i Q_i\succeq 0$.

			\item $\textup{conv}(S) = \mathbb{R}^n$. Suppose $S\ne\emptyset$, or equivalently $\P\cap\R_+^n\ne\emptyset$. Then $\conv(S)=\R^n$ if and only if $\P'\supseteq \R_+^n$, and by Proposition \ref{prop:diagonal_Rn} this is equivalent to existence of $u\ge 0$ such that $u^{\top}A\ge 0$ and is nonzero. 
			
			\item $\emptyset \subsetneq \textup{conv}(S) \subsetneq \mathbb{R}^n$: \blue{In this case} $\P\cap \R^n_+\ne\emptyset$. Therefore, using Proposition \ref{prop:conv} and Proposition \ref{prop:facets},  we obtain that $\textup{conv}(S)$ is given by finitely many aggregations where the aggregated constraint is convex, as each defining inequality of $\P'$ has the form $(\lambda^{(j)})^{\top}Ay<(\lambda^{(j)})^{\top}b$, for some $\lambda^{(j)}\ge 0$, corresponding to the following aggregation $$S_{\lambda^{(j)}}=\{x\in\R^n: \begin{bmatrix}
				x\\1
			\end{bmatrix}^{\top} \left(\sum_{i=1}^j \lambda_i^{(j)}Q_i\right)\begin{bmatrix}
				x\\1
			\end{bmatrix}< 0 \}.$$ 
			
			The leading $n\times n$ principal submatrix of $Q_{\lambda^{(j)}}\sum_{i=1}^j \lambda_i^{(j)}Q_i$ is PSD as $(\lambda^{(j)})^{\top}A\ge 0$ from Proposition \ref{prop:facets}, so $Q_{\lambda^{(j)}}$ has at most one negative eigenvalue and $S_{\lambda^{(j)}}$ is convex, which means $\lambda^{(j)}\in\Omega$. 
		\end{itemize}
	\end{proof}

	
	\section{Finite number of aggregations sufficient to obtain the convex hull}\label{sec:finite}

	
	We let $\Theta=\{\theta\in\R^m: \sum_{i=1}^m\theta_iQ_i\succeq 0 \}$ denote the set of linear combinations that gives rise to a PSD matrix. Our first observation is that one may always improve an aggregation $\lambda\in\Omega$ (where $S_{\lambda}\ne\R^n$) by elements in $\Theta$ as long as it still stays inside nonnegative orthant. \blue{Recall that $\Omega = \{\lambda \in \mathbb{R}^m_+\setminus\{0\} : \conv(S)\subseteq S_{\lambda} \textup{ and }Q_{\lambda} \textup{ has at most one negative eigenvalue}\}$. }
	
	\begin{proposition}\label{prop:improve}
		Assume $S\ne\emptyset$ and $\conv(S)\ne\R^n$. Let $\Omega_1=\Omega\setminus\{\lambda\in\Omega:S_{\lambda}=\R^n \}$. Let $\lambda\in\Omega_1$ and $\theta\in\Theta$ so that $\lambda'=\lambda+\theta\in\R^n_+\setminus\{0\}$. Then $\lambda'\in\Omega_1$ and $S_{\lambda'}\subseteq S_{\lambda}$.  
	\end{proposition}
	
	\blue{Note that Proposition \ref{prop:improve} is for arbitrary quadratics, and in particular does not assume HHC.}
	
	\begin{proof}
		Since $Q_{\lambda'}\succeq Q_{\lambda}$ we have $f_{\lambda'}\ge f_{\lambda}$ and $S_{\lambda'}\subseteq S$, which also implies $S_{\lambda'}\ne\R^n$ since $S_{\lambda}\ne\R^n$. The next step is to show $Q_{\lambda'}$ has exactly one negative eigenvalue. Since $Q_{\lambda'}\succeq Q_{\lambda}$, $Q_{\lambda'}$ cannot have more negative eigenvalues than $Q_{\lambda}$, due to Weyl's inequality on eigenvalues \cite{Horn1985matrix}. Since $Q_{\lambda}$ has at most one negative eigenvalue, same must be true for $Q_{\lambda'}$. On the other hand $Q_{\lambda'}$ is not PSD since $S\ne\emptyset$ and $S\subseteq S_{\lambda'}$. Thus it has exactly one negative eigenvalue. 
		
		It remains to show $\conv(S)\subseteq S_{\lambda'}$. We consider several cases based on whether $S_{\lambda'}$ and $S_{\lambda}$ are convex or unions of two disjoint convex sets. 
		
		\begin{itemize}
			\item $S_{\lambda'}$ is convex then we are done since $S\subseteq S_{\lambda'}$. 
			
			\item $S_{\lambda'}$ is union of two disjoint convex sets, i.e., $S_{\lambda'}=C_1'\cup C_2'$ where $C_1',C_2'$ are disjoint and convex. There are two subcases about $S_{\lambda}$. 
			\begin{itemize}
				\item $S_{\lambda}$ is convex. From Lemma \ref{lem:single_quadratic} $\conv(S_{\lambda'})=\R^n$. Since $S_{\lambda'}\subseteq S_{\lambda}$ we must have $S_{\lambda}=\R^n$, contradicting our assumption that $\lambda\in\Omega_1$. 
				
				\item $S_{\lambda}$ is union of two disjoint convex sets, i.e., $S_{\lambda}=C_1\cup C_2$ where $C_1,C_2$ are disjoint and convex. Since $\conv(S)\subseteq S_{\lambda}$, \blue{$\conv(S)$} must lie entirely in one of \blue{$C_1$ or $C_2$}, say $C_1$ upon relabeling, which means $\conv(S)\cap C_2=\emptyset$ since $C_1,C_2$ are disjoint. Since $S_{\lambda'}\subseteq S_{\lambda}$, upon relabeling assume $C_1'\subseteq C_1$, $C_2'\subseteq C_2$. This means $\conv(S)\cap C_2'=\emptyset$ and $\conv(S)\subseteq C_1'\subseteq S_{\lambda'}$. Hence $\lambda'\in \Omega_1$. 
			\end{itemize}
	\end{itemize}\end{proof}
	
	We now prove \blue{under the assumption that every triple of quadratics satisfies the PDLC condition, elements in $\Omega_1$ with support at most 2 describe the same set as all elements in $\Omega_1$.} The idea is to repeatedly improve along positive definite linear combinations to reduce the support of a good aggregation.

	\begin{proposition}\label{prop:2aggregation}
		Assume $S\ne\emptyset$ and $\conv(S)\ne\R^n$. \blue{Let} $\Omega_1=\Omega\setminus\{\lambda\in\Omega:S_{\lambda}=\R^n \}$. Furthermore assume for all distinct $i,j,k\in [m]$ there exist scalars $p_{ijk},q_{ijk},r_{ijk}\in\R$ such that $p_{ijk} Q_i+q_{ijk}Q_j+r_{ijk}Q_k\succ 0$. Let $\Omega_2=\{\lambda\in\Omega_1: |\{i:\lambda_i>0\}|\le 2 \}$. Then $\bigcap_{\lambda\in\Omega_2}S_{\lambda}=\bigcap_{\lambda\in\Omega_1}S_{\lambda}$. 
	\end{proposition}
	
	\begin{proof}
		Clearly $\bigcap_{\lambda\in\Omega_1}S_{\lambda}\subseteq\bigcap_{\lambda\in\Omega_2}S_{\lambda}$ since $\Omega_2\subseteq\Omega_1$. For reverse inclusion we show that for any $\lambda\in\Omega_1$ with $|\{i:\lambda_i>0 \}|\ge 3$ there exists $\lambda'\in\Omega_1$ with $|\{i:\lambda'_i>0 \}|<|\{i:\lambda_i>0 \}|$ with $S_{\lambda'}\subseteq S_{\lambda}$. Then repeatedly applying this subroutine whenever possible, we eventually get $\lambda''\in\Omega_1$ with $|\{i:\lambda''_i>0 \}|\le 2$ and $S_{\lambda''}\subseteq S_{\lambda}$. 
		
		Fix $i,j,k\in \{l:\lambda_l>0 \}$ and $p_{ijk},q_{ijk},r_{ijk}\in\R$ such that $p_{ijk} Q_i+q_{ijk}Q_j+r_{ijk}Q_k\succ 0$. 
		If $\lambda$ is a multiple of $v=p_{ijk}e_i+q_{ijk}e_j+r_{ijk}e_k$, we can perturb $(p_{ijk},q_{ijk},r_{ijk})$ so that $\lambda$ is not a multiple of $v=p_{ijk}e_i+q_{ijk}e_j+r_{ijk}e_k$ and $p_{ijk} Q_i+q_{ijk}Q_j+r_{ijk}Q_k\succ 0$. 
		Also note that $p_{ijk},q_{ijk},r_{ijk}$ cannot be all nonnegative, since otherwise we have that $S\subseteq S_v=\emptyset$.
		Now let $\alpha_0=\max\{\alpha>0:\lambda+\alpha v\in\R^m_+ \}$ and $\lambda'=\lambda+\alpha_0 v$. More explicitly $\alpha_0$ is the minimum between $\frac{\lambda_i}{-p_{ijk}},\frac{\lambda_j}{-q_{ijk}},\frac{\lambda_k}{-r_{ijk}}$ where we only consider the terms where denominator is positive. Then clearly $|\{i:\lambda'_i>0 \}|<|\{i:\lambda_i>0 \}|$, and $\lambda'\in \Omega_1,S_{\lambda'}\subseteq S_{\lambda}$ due to Proposition \ref{prop:improve}. 
	\end{proof}

	Now we study the structure of aggregations with fixed support of size two \blue{that contain $\conv(S)$, and show they are either empty or form one or two intervals, whose endpoints are the same as the intervals described in \cite{yildiran2009convex}. Thus} for the description of the convex hull it suffices to take either the two outermost endpoints, or two endpoints of the same interval. The key ingredient in our proof is the geometry of the set defined by two quadratic inequalities, which was studied in \cite{yildiran2009convex}. Here we list the results that are needed for our proof, and describe their implications. Note that the versions stated in our paper differ by a sign compared to \cite{yildiran2009convex}, as we study the set of points where the quadratic inequalities are negative (instead of positive). 
	
	In words, these results show that convex combinations of any two quadratics contain at most two intervals of matrices that have at most one negative eigenvalue. The endpoints of the intervals can be recognized as points where the rank drops. Furthermore, homogenized good aggregations lying in the same interval have an additional geometric property described in Lemma \ref{lem:yildiran_hyperplane}. 
	
	Now we formally state these results. Let $\nu(M)$ denote the number of negative eigenvalues of a matrix $M$. 
	
	\begin{theorem}[Lemma 2 of \cite{yildiran2009convex}]\label{thm:yildiran_intervals}
		Let $Q_1,Q_2$ be two symmetric matrices, and let $\Lambda=\{0\le \alpha\le 1: \nu(\alpha Q_1+(1-\alpha)Q_2)=1 \}$. If $\Lambda\ne \emptyset$, then there exists $n_c\in\{1,2\}$ and $\Lambda=\bigcup_{1\le j\le n_c}\{\I_j\}$, where each $\I_j$ is a closed interval of $[0,1]$ and $\I_j,\I_k$ are disjoint if $j\ne k$. Furthermore, the endpoints of $\I_j$ are real roots of $\det(\alpha Q_1+(1-\alpha)Q_2)=0$, known as generalized eigenvalues (GEVs) of $Q_1$ and $Q_2$. 
	\end{theorem}
	
	\begin{lemma}[Lemma 7 of \cite{yildiran2009convex}]\label{lem:yildiran_hyperplane}
		Let $\I$ be one interval in the previous theorem. Then there exists a linear hyperplane $L$ that does not intersect $\{x:x^{\top}(\alpha Q_1+(1-\alpha)Q_2)x<0 \}$ for any $\alpha\in\I$. 
	\end{lemma}
	
	Using these results we prove the following proposition about aggregations of homogeneous quadratics with fixed support of size two: if all aggregations contain a given set $S_0$ and one aggregation with one negative eigenvalue contains $\conv(S_0)$ (so that this aggregation is good), then all aggregations in the same interval must also contain $\conv(S_0)$, i.e. all aggregations in the same interval are good. \blue{The main idea is to use the hyperplane $L$ as promised by Lemma \ref{lem:yildiran_hyperplane}, and show that $\conv(S_0)$ must lie entirely in one of the halfspaces separated by $L$, and be disjoint from the other halfspace. }

	\begin{proposition}\label{prop:whole_interval}
		Let $Q_1,Q_2$ be $(n+1)\times (n+1)$ symmetric matrices. Let $S_0\subseteq \R^{n+1}$ and assume $\emptyset\subsetneq \conv(S_0)\subsetneq \R^{n+1}$. Let $\Lambda=\{0\le \alpha\le 1: \nu(\alpha Q_1+(1-\alpha)Q_2)=1 \}$. Suppose $\Lambda\ne\emptyset$ and $S_0\subseteq \{\hx\in\R^{n+1}:\hx^{\top}(\alpha Q_1+(1-\alpha)Q_2)\hx<0 \}$ for all $\alpha\in\Lambda$. 
		
		Based on Theorem \ref{thm:yildiran_intervals} we have $\Lambda=\bigcup_{1\le j\le n_c}\{\I_j\}$, $n_c\in\{1,2\}$, where each $\I_j$ is a closed interval of $[0,1]$ and $\I_j,\I_k$ are disjoint if $j\ne k$. Fix $1\le j\le n_c$ and assume $$\conv(S_0)\subseteq \{\hx\in\R^{n+1}:\hx^{\top}(\alpha_0 Q_1+(1-\alpha_0)Q_2)\hx<0 \}$$ for some $\alpha_0\in \I_j$. Then $$\conv(S_0)\subseteq \{\hx\in\R^{n+1}:\hx^{\top}(\alpha Q_1+(1-\alpha)Q_2)\hx<0 \}$$ for all $\alpha\in \I_j$. 
	\end{proposition}
	
	\blue{
		
		\begin{proof}
			From Theorem \ref{thm:SCC}, the set $\{\hx\in\R^{n+1}:\hx^{\top}(\alpha_0 Q_1+(1-\alpha_0)Q_2)\hx<0 \}$ is an SCC, i.e., a union of two disjoint open convex cones that are symmetric reflections of each other across the origin. Let $\{\hx\in\R^{n+1}:\hx^{\top}(\alpha_0 Q_1+(1-\alpha_0)Q_2)\hx<0 \}=C_{\alpha_0}^+\cup C_{\alpha_0}^-$, where $C_{\alpha_0}^+, C_{\alpha_0}^-$ are the disjoint open convex cones. Since $\conv(S_0)\subseteq \{\hx\in\R^{n+1}:\hx^{\top}(\alpha_0 Q_1+(1-\alpha_0)Q_2)\hx<0 \}$, it is fully contained in one of the convex cones and disjoint from the other. Upon relabeling we may assume $\conv(S_0)\subseteq C_{\alpha_0}^+$ and $\conv(S_0)\cap C_{\alpha_0}^-=\emptyset$. 
			
			From Lemma \ref{lem:yildiran_hyperplane}, there exists linear hyperplane $L$ that does not intersect $\{\hx\in\R^{n+1}:\hx^{\top}(\alpha Q_1+(1-\alpha)Q_2)\hx<0 \}$ for all $
			\alpha\in\I_j$. Let $L^+,L^-$ be the two open halfspaces separated by $L$. Upon relabeling we may assume $C_{\alpha_0}^+\subseteq L^+,C_{\alpha_0}^-\subseteq L^-$. 
			
			Now fix arbitrary $\alpha'\in \I_j$. The set $\{\hx\in\R^{n+1}:\hx^{\top}(\alpha' Q_1+(1-\alpha')Q_2)\hx<0 \}$ is an SCC. Let $\{\hx\in\R^{n+1}:\hx^{\top}(\alpha' Q_1+(1-\alpha')Q_2)\hx<0 \}=C_{\alpha'}^+\cup C_{\alpha'}^-$, where $C_{\alpha'}^+, C_{\alpha'}^-$ are the disjoint open convex cones. Note $L\cap \{\hx\in\R^{n+1}:\hx^{\top}(\alpha' Q_1+(1-\alpha')Q_2)\hx<0 \}=\emptyset$, which means each of $C_{\alpha'}^+, C_{\alpha'}^-$ is fully contained in one of the open halfspaces separated by $L$. Upon relabeling we may assume $C_{\alpha'}^+\subseteq L^+,C_{\alpha'}^-\subseteq L^-$. 
			
			Since $\conv(S_0)\subseteq C_{\alpha_0}^+\subseteq L^+$ and $L^+\cap L^-=\emptyset$, we have $\conv(S_0)\cap L^-=\emptyset$ and hence $S_0\cap L^-=\emptyset$. Therefore $S_0\cap C_{\alpha'}^-=\emptyset$ as $C_{\alpha'}^-\subseteq L^-$, and $S_0\subseteq C_{\alpha'}^+$. Since $C_{\alpha'}^+$ is convex, we have $\conv(S_0)\subseteq C_{\alpha'}^+\subseteq \{\hx\in\R^{n+1}:\hx^{\top}(\alpha' Q_1+(1-\alpha')Q_2)\hx<0 \}$ as desired. 
		\end{proof}
		
	}
	

	Recall that $S=\{x\in\R^n: x^{\top}A_i x+2b_i^{\top}x+c_i<0,i\in [m] \}$ and $\Omega$ is the set of good aggregations which have at most one negative eigenvalue and contain $\conv(S)$. Proposition \ref{prop:whole_interval} implies the following result about pairwise aggregations that contain the convex hull.

	\begin{proposition}\label{prop:pairwise_endpoints}
		Assume $S\ne\emptyset$, $\conv(S)\ne\R^n$ and hidden hyperplane convexity holds for the associated quadratic map $\hf$, so that Theorem \ref{thm:more_quadratics} holds and $\conv(S)=\bigcap_{\lambda\in\Omega}S_{\lambda}$. Fix $i,j\in [m]$ and let $$\Omega_{ij}=\{\lambda\in\Omega: \lambda_k=0, \forall k\notin \{i,j\} \}$$ be aggregations in $\Omega$ that have support in $\{i,j\}$. Then either $\Omega_{ij}=\emptyset$, or there exists $\lambda',\lambda''\in \Omega_{ij}$ such that  $\bigcap_{\lambda\in\Omega_{ij}}S_{\lambda}=S_{\lambda'}\cap S_{\lambda''}$, where $\lambda',\lambda''$ can be written as $\lambda'=\alpha'e_i+(1-\alpha')e_j,\lambda''=\alpha''e_i+(1-\alpha'')e_j$, where $\alpha',\alpha''$ are roots of $\det (\alpha Q_i+(1-\alpha)Q_j)=0$. 
	\end{proposition} 
	
	\begin{proof}
		\blue{Throughout this proof, we reparametrize aggregations with support $\{i,j\}$ by the unit length interval $[0,1]$, and view each $0\le\alpha\le 1$ as the aggregation $\alpha e_i+(1-\alpha) e_j$.}
		
		Assume $\Omega_{ij}\ne\emptyset$. Clearly $\bigcap_{\lambda\in\Omega_{ij}}S_{\lambda}\subseteq S_{\lambda'}\cap S_{\lambda''}$ since $\lambda',\lambda''\in \Omega_{ij}$. For reverse inclusion, from Theorem \ref{thm:yildiran_intervals}, $\{\lambda: \lambda_k=0,\forall k\notin \{i,j\}\}$ where $Q_{\lambda}$ has at most one negative eigenvalues forms one or two closed intervals \blue{(contained in the reparametrized $[0,1]$ interval)}. Furthermore, each interval either lies entirely in $\Omega_{ij}$ or is disjoint from it, by applying Proposition \ref{prop:whole_interval} to $S_0=\{(x,1)\in\R^{n+1}:x\in S\}$ and observing that $S_0\subseteq S_{\lambda}\blue{\times \{1\}}$ for all nonzero $\lambda\ge 0$.  
		Thus $\Omega_{ij}$ is also one or two closed intervals whose endpoints are GEVs of $Q_i$ and $Q_j$. We let $\lambda',\lambda''$ be the two outermost endpoints of $\Omega_{ij}$, and observe that any $\lambda\in\Omega_{ij}$ is a nonnegative combination of $\lambda'$ and $\lambda''$ and therefore \blue{$S_{\lambda'}\cap S_{\lambda''}\subseteq S_{\lambda}$}, which means $S_{\lambda'}\cap S_{\lambda''}\subseteq \bigcap_{\lambda\in\Omega_{ij}}S_{\lambda}$. 
	\end{proof}
	
	Proposition~\ref{prop:2aggregation} and Proposition ~\ref{prop:pairwise_endpoints} together imply Theorem~\ref{thm:finite}. 
	
	\begin{proof}[Proof of Theorem~\ref{thm:finite}]
		Let $\Omega_2=\{\lambda\in\Omega_1: |\{i:\lambda_i>0\}|\le 2 \}$. Then from Proposition~\ref{prop:2aggregation}, $\bigcap_{\lambda\in\Omega_1}S_{\lambda}=\bigcap_{\lambda\in\Omega_2}S_{\lambda}$. Now $\Omega_2=\bigcup_{i\ne j}\Omega_{ij}\setminus\{\lambda\in\Omega:S_{\lambda}=\R^n \}$. Therefore Theorem~\ref{thm:finite} follows after applying Proposition \ref{prop:pairwise_endpoints} to all $\{i,j\}\subseteq [m]$ and removing any $\lambda'$ where $S_{\lambda'}=\R^n$, which does not change the convex hull. 
	\end{proof}

	\section{Results for closed inequalities}\label{sec:closed}

	\begin{proof}[Proof of Theorem \ref{prop:agg_interior}]
		Recall that $T=\{x:f_i(x)\le 0,i\in [m]\}$ is the set defined by closed inequalities and $G=\textup{int}(\overline{\conv(T)})$.  We remind the reader that $S \subseteq G$ and we assume $G \neq \emptyset$. 
		
		Take any $y\notin G$. Since $G$ is convex and open, there exists $\alpha\in\R^n$ such that $G\subseteq \{x:\alpha^{\top}x>\alpha^{\top}y\}$. Let $H=\{x: \alpha^{\top}x=\alpha^{\top}y\}\subseteq\R^n$ be the separating affine hyperplane and $\hH=\{(x,x_{n+1}):\alpha^{\top}x=(\alpha^{\top} y)x_{n+1}\}\subseteq \R^{n+1}$ be its homogenization. Let $\hH_+=\{(x,x_{n+1}):\alpha^{\top}x \ge(\alpha^{\top} y)x_{n+1}\},\hH_-=\{(x,x_{n+1}):\alpha^{\top}x \le (\alpha^{\top} y)x_{n+1}\}$ be the closed halfspaces created by $\hH$. Then we have $\overline{\conv(T)}\times \{1\} \subseteq \hH_+$. 
		
		Note that Lemma \ref{lem:homog_separation} still applies since $\conv(S) \neq \mathbb{R}^n$. Thus, since $y \notin S$ (because $S \subseteq G$) we have that $S^h \cap  \hat{H} = \emptyset$, i.e., 
		$\{\hx\in\hH: \hx^{\top}Q_i\hx<0, i \in [m]\}=\emptyset$. 
		Proceeding in the same way \blue{as} in the proof \blue{of} Theorem~\ref{thm:more_quadratics}, we use the fact that $Q_i$\blue{'s} satisfy hidden hyperplane convexity, to show that there exists a nonzero vector $\lambda\in\R_+^m$ such that $Q_{\lambda}$ is PSD on $\hH$. By our assumption we also have $Q_{\lambda}\ne 0$.

		
		By the Interlacing Theorem, $Q_{\lambda}$ has at most one negative eigenvalue, and it cannot be PSD otherwise $T\times \{1\} \subseteq \ker Q_{\lambda}$, which would imply $G=\emptyset$. This means $Q_{\lambda}$ has exactly one negative eigenvalue. Therefore $S_{\lambda}$ consists of one or two disjoint open convex sets. Let $T_{\lambda}=\{x:\sum_{i=1}^m \lambda_i f_i(x)\le 0\}$ denote the set defined by the same aggregation with closed inequalities. It is clear that $T_{\lambda}$ is closed, $T\subseteq T_{\lambda}$ and $S_{\lambda}=\textup{int}( T_{\lambda})$. 
		
		We now show that $\overline{\conv(T)}\subseteq T_{\lambda}$ which then implies that $G\subseteq S_{\lambda}$. If $S_{\lambda}$ is convex then we are done as in this case $T_{\lambda}$ is convex and closed. Suppose $S_{\lambda}$ consists of two disjoint open convex connected components. We write $S_{\lambda}=(S_{\lambda})_+\cup (S_{\lambda})_-$ where $(S_{\lambda})_+ ,(S_{\lambda})_-$ are disjoint open convex sets. Since $Q_{\lambda}$ is PSD on $\hH$, these two sets lie in different sides of $\hH$, and upon relabeling we assume $(S_{\lambda})_+\subseteq \hH_+, (S_{\lambda})_-\subseteq \hH_-$. 
		
		Let $(T_{\lambda})_+= \overline{(S_{\lambda})_+},(T_{\lambda})_-= \overline{(S_{\lambda})_-}$. Then it is clear that $T_{\lambda}=(T_{\lambda})_+\cup (T_{\lambda})_-, (T_{\lambda})_+\cap (T_{\lambda})_-\subseteq \hH$, and $(T_{\lambda})_+\subseteq \hH_+, (T_{\lambda})_-\subseteq \hH_-$. 
		
		Recall $T\times \{1\}\subseteq \hH_+$. We claim this implies $T\times \{1\}\subseteq (T_{\lambda})_+$, which completes the proof as $(T_{\lambda})_+$ is closed and convex. Suppose otherwise, then $(T\times \{1\})\cap ((T_{\lambda})_-\setminus (T_{\lambda})_+)\ne\emptyset$, but $(T_{\lambda})_-\setminus (T_{\lambda})_+\subseteq \hH_-\setminus \hH$ which is disjoint from $\hH_+$. 
	\end{proof}


	
	
	
	
	
	\section{Acknowledgements}
	We would like to thanks Gonzalo Mu\~{n}oz and Felipe Serrano for many useful discussions on this topic. 
	
	\bibliographystyle{plain}
	\bibliography{references}

\end{document}